\documentclass{amsart}
\usepackage{xy,cleveref,enumitem,verbatim,amssymb}
\xyoption{all}
\title{Computation of the classifying ring of formal modules}
\author{A. Salch}
\DeclareMathOperator{\End}{{\rm End}}
\DeclareMathOperator{\Rees}{{\rm Rees}}
\DeclareMathOperator{\Symm}{{\rm Sym}}
\DeclareMathOperator{\Spec}{{\rm Spec}}
\DeclareMathOperator{\Ext}{{\rm Ext}}
\DeclareMathOperator{\im}{{\rm im\:}}
\DeclareMathOperator{\coker}{{\rm coker}}
\DeclareMathOperator{\id}{{\rm id}}
\DeclareMathOperator{\Sym}{{\rm Sym}}
\DeclareMathOperator{\Aut}{{\rm Aut}}

\newcommand{\customcomment}[1]{}

\theoremstyle{plain}
\newtheorem{prop}{Proposition}[subsection]
\newtheorem{theorem}[prop]{Theorem}
\newtheorem{corollary}[prop]{Corollary}
\newtheorem{setup for thm}[prop]{Setup for theorem}

\newtheorem*{unn motivating questions}{Motivating Questions}
\newtheorem{lemma}[prop]{Lemma}

\newtheorem{motivating questions}[prop]{Motivating Questions}
\newtheorem{definition}[prop]{Definition}
\newtheorem{definition-proposition}[prop]{Definition-Proposition}
\newtheorem{definition-theorem}[prop]{Definition-Theorem}

\newtheorem{running notations}[prop]{Running notations}

\theoremstyle{definition}
\newtheorem{remark}[prop]{Remark}
\newtheorem{intuitive idea}[prop]{Intuitive idea}

\newtheorem{example}[prop]{Example}

\newtheorem{convention}[prop]{Conventions}
\newtheorem{conventions}[prop]{Conventions}

\newtheorem*{sketch of proof}{Sketch of proof}

\begin{document}
\begin{abstract}
In this paper, we develop general machinery for computing the classifying ring $L^A$ of one-dimensional formal $A$-modules, for various commutative rings $A$. 
We then apply the machinery to obtain calculations of $L^A$ for various number rings and cyclic group rings $A$. This includes the first full calculations of the ring $L^A$ in cases in which it fails to be a polynomial algebra. We also derive consequences for the solvability of some lifting and extension problems. 
\end{abstract}
\maketitle
\tableofcontents

\section{Introduction and review of some known facts}

\subsection{Introduction}

This paper 
is about the computation of the classifying rings $L^A$ 
of formal $A$-modules. We offer some relevant definitions: when $A$ is a commutative ring, a {\em formal $A$-module}
is a formal group law $F$ over a commutative $A$-algebra $R$, which is additionally equipped with a ring map $\rho: A \rightarrow \End(F)$ such that $\rho(a)(X) \equiv aX$ modulo $X^2$.
Higher-dimensional formal modules exist and arise naturally from abelian varieties of dimension $>1$, but in this paper we restrict our attention to the moduli of formal modules of dimension $1$. From now on, all formal groups and formal modules are implicitly assumed to be one-dimensional (see Conventions \ref{nu convention} for this and other conventions). 

Formal modules arise in number theory and arithmetic geometry, for example, in Lubin and Tate's famous $p$-adic version \cite{MR0172878} of the Kronecker-Weber theorem; 
in Drinfeld's elliptic modules \cite{MR0384707}; in Drinfeld's $p$-adic symmetric domains \cite{MR0422290}, which are deformation spaces of certain formal modules; and as the formal parts of Barsotti-Tate modules, as in \cite{MR1393439}. Formal $A$-modules also arise in algebraic topology, by using the natural map from the moduli stack of formal $A$-modules to the moduli stack of formal groups to detect certain classes in the cohomology of the latter, particularly in order to resolve certain differentials in spectral sequences used to compute the Adams-Novikov $E_2$-term and stable homotopy groups of spheres; for example, see \cite{MR745362}, \cite{pearlmanthesis}, \cite{formalmodules4}, and \cite{cmah7}.

A straightforward algebraic argument (see \cite{MR0384707}) shows that there exists a classifying ring $L^A$ for formal $A$-modules, i.e., 
a commutative $A$-algebra $L^A$ such that $\hom_{A-alg}(L^A, R)$ is in natural bijection with the set of formal $A$-modules over $R$. It remains an open problem to explicitly compute $L^A$ for a general commutative ring $A$. The ring $L^A$ has been computed for only a few classes of rings $A$, described in the next paragraph. In each case, {\em the ring $L^A$ has turned out to be a polynomial $A$-algebra}.
One reason this matters is that certain qualitative features of formal $A$-modules follow as a consequence of the calculation of $L^A$.
As an example, in \cite{MR0384707} Drinfeld obtained the following results, under the assumption that $A$ is the ring of integers in a nonarchimedean local field:
\begin{description}
\item[Extension] Every formal $A$-module $n$-bud extends to a formal $A$-module.
(A {\em formal module $n$-bud} is what one gets by reducing the entire definition of a formal module modulo $(X,Y)^{n+1}$, so that $F(X,Y)$ is a power series in $R[[X,Y]]/(X,Y)^{n+1}$, and it is only required to satisfy the axioms for a formal module modulo $(X,Y)^{n+1}$.)
\item[Lifting] 
If $R$ is a commutative $A$-algebra and $I$ is an ideal of $R$, then every formal $A$-module over $R/I$ is the modulo-$I$ reduction
of a formal $A$-module over $R$.
\end{description}
These two qualitative features of formal $A$-modules follow immediately from $L^A$ being a polynomial $A$-algebra, and that is {\em how these properties are proven}: one calculates $L^A$, and the extension and lifting properties follow as a consequence.

Consequently, it is of some value to have calculations of the ring $L^A$ for various rings $A$, particularly Dedekind domains, since formal $A$-modules of interest in number theory often have $A$ the ring of integers in a (local or global) number field.
There are three cases of rings $A$ for which $L^A$ is already calculated:
\begin{enumerate}
\item In \cite{MR0393050}, M. Lazard proved that 
\begin{equation}\label{lazard iso} L^{A} \cong A[x_1, x_2, \dots],\end{equation}
a polynomial algebra on countably infinitely many generators, when $A = \mathbb{Z}$. As a consequence, the ring $L^{\mathbb{Z}}$ is often called the Lazard ring. 
\item In \cite{MR0384707}, V. Drinfeld proved that there is also an isomorphism of the form \eqref{lazard iso} when $A$ is the ring of integers in a local nonarchimedean field (e.g. the ring of integers in a finite extension of $\mathbb{Q}_p$).
\item 
Finally, M. Hazewinkel proved in \cite{MR506881} that there is an isomorphism of the form \eqref{lazard iso} when $A$ is a discrete valuation ring or a global number ring of class number one.
\end{enumerate}

One cannot expect an isomorphism of the form \eqref{lazard iso} for an arbitrary (global) number ring $A$. As far as the author knows, this was first observed by Hazewinkel, who showed that, in the case when $A$ is the ring of integers in the field $\mathbb{Q}(\sqrt[4]{-18})$, the sub-$A$-module of $L^A$ consisting of elements of degree $2$ is not a free $A$-module, and consequently $L^A$ cannot be a polynomial $A$-algebra. This observation appears in 21.3.3A of \cite{MR506881}, but no attempt is made there to compute the ring $L^A$.

In the present paper we compute the ring $L^A$ for certain classes of commutative ring $A$. 
These are the first known full computations of $L^A$ in which $L^A$ fails to be polynomial. Specifically, the computations we make are as follows:
\begin{itemize}[leftmargin=*]
\item In Theorem \ref{main thm in hereditary case},
we prove the following: let $A$ be a torsion-free commutative ring, and let $S$ be a set of prime numbers such that the ring $A[S^{-1}]$ is hereditary.
(If, for example, $A$ is already hereditary,
then we can let $S$ be the empty set.)

Then the commutative graded ring $L^A$
is, after inverting $S$, isomorphic to a tensor product 
of (graded) Rees algebras\footnote{The {\em Rees algebra} of an ideal $I$ in a commutative ring $A$, written $\Rees_A(I)$, is the commutative $A$-algebra $\coprod_{n\geq 0} I^n\{ t^n\} \subseteq A[t]$. The notation $\Rees^n_A(I)$ denotes $\Rees_A(I)$ as a graded $A$-algebra, graded so that $I^m$ is in degree $mn$. 

Similarly, here and throughout this paper, $\Symm_A^n(I)$ does {\em not} denote the $n$th symmetric power of $I$. Rather, it denotes the symmetric $A$-algebra of $A$, equipped with the grading in which the $m$th symmetric power of $I$ is concentrated in degree $mn$. See Definition \ref{def of susp symm and susp rees}.}, and also isomorphic to a tensor product of graded symmetric algebras:
\begin{align*} L^A[S^{-1}] &\cong \left( \Rees^2_{A}(I^A_2) \otimes_{A} \Rees_{A}^4(I^A_3) \otimes_{A} \Rees_{A}^6(I^A_4) \right. \\ &\ \ \ \ \ \ \ \ \ \left.\otimes_{A}\Rees_{A}^8(I^A_5) \otimes_{A}\Rees_{A}^{10}(I^A_6) \otimes_{A} \dots\right) [S^{-1}] \\
&\cong \left( \Symm^2_{A}(I^A_2) \otimes_{A} \Symm_{A}^4(I^A_3) \otimes_{A} \Symm_{A}^6(I^A_4) \right. \\ &\ \ \ \ \ \ \ \ \ \left.\otimes_{A}\Symm_{A}^8(I^A_5) \otimes_{A}\Symm_{A}^{10}(I^A_6) \otimes_{A} \dots\right) [S^{-1}] 
\end{align*}
where $I^A_n$ is the ideal in $A$ generated by $\nu(n)$ and by all elements of the form $a^n-a$, and where $\nu(n)$ is defined as follows:
\begin{itemize}
\item $\nu(n)$ is defined only for integers $n>1$,
\item $\nu(n)=1$ if $n$ is not a prime power,
\item and, if $n$ is a power of a prime number $p$, then $\nu(n)$ is defined to be $p$.
\end{itemize}

In particular, the ring $L^A[S^{-1}]$ is a polynomial $A$-algebra if and only if each of the ideals $I^A_n$ is nonprincipal in $A[S^{-1}]$.
\item In Theorem \ref{number ring iso 0}, we prove the following:
let $A$ be the ring of integers in a finite extension $K/\mathbb{Q}$,
let $1, \alpha_1, \dots , \alpha_j$ be a $\mathbb{Z}$-linear basis for $A$,
and let $J^A_n$ be the ideal $(\nu(n), \alpha_1^n-\alpha_1, \alpha_2^n - \alpha_2, \dots ,\alpha_j^n - \alpha_j)$ of $A$.
Let $P$ denote the set of integers $>1$ which are prime powers, and let $R$ denote the set of integers $>1$ which are not prime powers.
Then we have isomorphisms of commutative graded $A$-algebras:
\begin{align*}
 L^A
  &\cong \left(\bigotimes^{n\in P} \Rees_A^{2n-2}(J^A_n)\right)
 \otimes_A A[x_{n-1}: n\in R] \\
  &\cong \left(\bigotimes^{n\in P} \Symm_A^{2n-2}(J^A_n)\right)
 \otimes_A A[x_{n-1}: n\in R] ,
\end{align*}
with $x_{n-1}$ in degree $2(n-1)$, and with all tensor products taken over the ring $A$.
To be clear, the notation $A[x_{n-1}: n\in R]$ denotes the polynomial $A$-algebra on the set of indeterminates $\{ x_{n-1}: n\in R\}$. 
\item For quadratic number rings, we prove Theorem \ref{quadratic case}:
let $K$ be a quadratic extension of the rational numbers,
and let $A = \mathbb{Z}[\alpha]$ be the ring of integers of $K$.
Let $\Delta$ denote the discriminant of $K/\mathbb{Q}$.
Let $R$ be the set of prime numbers $p$ which divide $\Delta$
and which have the property that $J^A_{p^m} = (p, \alpha^{p^m}-\alpha)$ is nonprincipal for some positive integer $m$.
Let $S$ be the set of integers $>1$ which are
not powers of primes contained in $R$.
Then we have an isomorphism of commutative graded $A$-algebras:
\begin{align*} L^A &\cong A[ x_{n-1}: n\in S] \otimes_A \\
 &\ \ \ \ \ \ \ \ \ \bigotimes^{p\in R} \left( \Rees_A^{2p-2}(J^A_{p}) \otimes_A \Rees_A^{2p^2-2}(J^A_{p^2}) \otimes_A \Rees_A^{2p^3-2}(J^A_{p^3}) \otimes_A \dots \right) \end{align*}
with each polynomial generator $x_{n-1}$ in degree $2(n-1)$, and with all the tensor products taken over $A$.

Consequently, 
we have an isomorphism of commutative graded $A[R^{-1}]$-algebras
$ L^A[R^{-1}]
\cong A[R^{-1}][x_1, x_2, \dots ],$
with each $x_i$ in degree $2i$. 
\item 
As an example computation, in Theorem \ref{fourth root of -18} 
we fully work out the ring $L^A$ in the case where $A$ is the ring of integers in the number field $\mathbb{Q}(\sqrt[4]{-18})$. This was Hazewinkel's original example of a number ring $A$ in which $L^A$ could not possibly be a polynomial ring (but Hazewinkel's computation stopped at grading degree $2$). The full result is:
let $S$ denote the set of all integers $>1$
which are not powers of $2$ or of $3$.
Then we have an isomorphism of commutative graded $A$-algebras
\begin{align*}
 L^A 
   & \cong A[ x_{n-1}: n\in S] \otimes_A 
    A[x_1,y_1]/(2x_1-(\alpha^2-\alpha)y_1) \\ 
   & \ \ \ \ \ \otimes_A 
    \bigotimes^{m\geq 2} \left( A[x_{2^m-1},y_{2^m-1}]/(2x_{2^m-1}-\alpha y_{2^m-1})\right) \\
   & \ \ \ \ \ \otimes_A
    \bigotimes^{m\geq 1} \left( A[x_{3^m-1},y_{3^m-1}]/(3x_{3^m-1}-\alpha y_{3^m-1})\right),
\end{align*}
where $\alpha = \sqrt[4]{-18}\in A$, 
where the polynomial generators $x_{i}$ and $y_i$ are each 
in degree $2i$, and where all the tensor products are taken over $A$.

Consequently, the ring $L^A$ is not isomorphic to a polynomial algebra, but we have an isomorphism of commutative graded $A[\frac{1}{6}]$-algebras
 $L^A\left[\frac{1}{6}\right]
  \cong A\left[\frac{1}{6}\right][x_1, x_2, \dots ],$
with each $x_i$ in degree $2i$.
\item 
Let $C_m$ be the cyclic group of order $m$, 
let $P$ be the set of integers $>1$ which are
prime powers relatively prime to $m$,
and let $S$ be the set of integers $>1$ not contained in $P$. 
Write $R$ for the group $\mathbb{Z}\left[\frac{1}{m}\right]$-algebra $\mathbb{Z}\left[\frac{1}{m}\right][C_m]$ of $C_m$.
Theorem \ref{group ring computation} establishes an isomorphism of graded rings
\begin{align*} L^{\mathbb{Z}[C_m]}\left[\frac{1}{m}\right] &\cong 
 \bigotimes^{n\in P} \left( R[x_{n-1},y_{n-1}]/(\nu(n)x_{n-1} - (1 - \sigma)y_{n-1})\right) \\    & \ \ \ \ \ \otimes_{R}
 R[x_{n-1}: n\in S],\end{align*}
where $\sigma$ denotes a generator of $C_m$, where  
the polynomial generators $x_{n-1}$ and $y_{n-1}$ are each 
in degree $2(n-1)$, and where the tensor products are all taken over $R$.
\item In each of the above cases, $L^A$ is a tensor product of Rees rings, so even when $L^A$ fails to be a polynomial algebra, {\em $L^A$ is still a subalgebra of a polynomial algebra.} Furthermore, in each of the above cases, $L^A$ is also a symmetric algebra on a projective module. As a consequence, we get Corollary \ref{structure cor in hereditary case}, a generalization of Drinfeld's lifting and extension theorems: for all of the hereditary rings $A$ described above, every formal $A$-module $n$-bud over a commutative $A$-algebra extends to a formal $A$-module. Furthermore, if $R$ is a commutative $A$-algebra and $I$ is an ideal of $R$, then every formal $A$-module over $R/I$ is the modulo-$I$ reduction
of a formal $A$-module over $R$. 
\end{itemize}

It is well-known that the classifying ring $L^AB$ of strict isomorphisms\footnote{Here is a quick review of the relevant ideas; for a fuller exposition, see \cite{MR506881}. Given formal $A$-modules $F,G$ over some $A$-algebra $R$, an isomorphism $f(X)\in R[[X]]$ of formal $A$-modules is said to be {\em strict} if $f(X) \equiv X \mod X^2$. There is an evident functor from commutative $A$-algebras to sets, given by sending each commutative $A$-algebra $R$ to the set of triples $(F,G,f)$, where $F,G$ are formal $A$-modules over $R$, and where $f: F \rightarrow G$ is a strict isomorphism. This functor is co-representable, and the traditional notation for its co-representing $A$-algebra is $L^AB$. 

While it is generally a subtle matter to work out the structure of the classifying ring $L^A$---this is the subject of the present paper!---it is much more straightforward, given $L^A$, to describe the structure of $L^AB$, since it is known that $L^AB$ is isomorphic to a polynomial algebra $L^A[b_1, b_2, \dots]$ over $L^A$, for {\em any} $A$.} of formal $A$-modules is always a polynomial algebra over $L^A$, and the Hopf algebroid $(L^A,L^AB)$ is isomorphic to $(L^A,L^A[b_1, b_2, \dots ])$. The stack associated to the groupoid scheme $(\Spec L^A, \Spec L^AB)$ is the moduli stack of formal $A$-modules, so the reader who is so inclined can regard the computations in this paper as computations of presentations for this moduli stack.

Producing these computations of $L^A$ for various rings $A$ requires some preliminary work. 
In \cref{u-homology section} we define a certain homology theory on rings, ``$U$-homology.'' Proposition \ref{u-homology detects injectivity of fundamental functional} shows that, in homological degrees $0$ and $1$, $U$-homology is the obstruction to $L^A$ being a polynomial algebra. These and other general properties of $U$-homology are worked out in \cref{u-homology general properties subsection}.

In \cref{calculation of U-homology} we use a theorem of Pirashvili, and a comparison to Hochschild homology with twisted coefficients, to show that $U$-homology is, in fact, {\em acyclic}: it vanishes in all positive degrees! 

In \cref{global consequences subsection}, we derive the various consequences of these results on $U$-homology. Most importantly, Corollary \ref{gen thm cor} establishes that, for a torsion-free ring $A$ such that the ideal $I_n^A = (\nu(n),a^n-a\ \forall a\in A)\subseteq A$ is a projective $A$-module for all integers $n>1$, $L^A$ is a tensor product of suspended Rees algebras.
We then prove Theorem \ref{main thm in hereditary case}, which similarly identifies $L^A[S^{-1}]$ for any torsion-free commutative ring $A$ and any multiplicatively-closed subset $S\subseteq A$ such that $A[S^{-1}]$ is hereditary. All of the above calculations of $L^A$ are enabled by Theorem \ref{main thm in hereditary case}.


The author originally posted an earlier version of this paper as a preprint with the title ``Computation of the classifying ring of formal groups with complex multiplication.'' The change from that title to the current one was because ``formal modules'' is a shorter and less old-fashioned way to say ``formal groups with complex multiplication.'' To avoid any possible confusion, we point out that a different preprint of the author's, ``Structure and cohomology of moduli of formal modules'' \cite{cmah1}, in an earlier version had the similar-sounding title ``The structure of the classifying ring of formal groups with complex multiplication,'' and some material that appeared in early versions of \cite{cmah1} has been moved into this paper (and removed from \cite{cmah1}). Despite the similarity in titles, the present paper does very different things from \cite{cmah1}: the present paper is concerned with the calculation of the classifying ring $L^A$ of formal $A$-modules for general commutative rings $A$, particularly hereditary rings $A$, while \cite{cmah1} proves some qualitative results about $L^A$ (e.g. behavior of $L^A$ with respect to completion of $A$) but is mostly concerned with explicit calculations in the cohomology of the moduli stack of formal $A$-modules, i.e., $\Ext$-groups in the category of comodules over the Hopf algebroid $(L^A,L^AB)$. 

The author is grateful to Doug Ravenel for teaching him a great deal about formal modules and homotopy theory when the author was a graduate student. The author also found the Sage and Magma computer algebra systems, \cite{sagemath66} and \cite{MR1484478}, quite helpful while preparing the number-theoretic results in \cref{number rings subsection}. Finally, the author is most grateful to the editor who handled this paper, Eric Friedlander, and to the anonymous referee for their tremendous persistence and patience with an author who took a bafflingly long time to make revisions.

\begin{conventions}\label{nu convention} Here are a few conventions which are in force throughout this paper.
\begin{itemize}
\item Given a ring $A$, we write $A\{ x\}$ for the free $A$-module on a generator $x$.
\item All graded rings considered in this paper are $\mathbb{N}$-graded, and concentrated in even degrees. Consequently, there is no question about whether the Koszul sign convention is in effect: all our commutative graded rings are genuinely commutative, i.e., commutative as ungraded rings.
\item For every commutative ring $A$, we equip the classifying ring $L^A$ of formal $A$-modules with the grading in which the homogeneous elements of degree\linebreak $2d-2$ are those elements which parameterize the total degree $d$ terms in a formal $A$-module, regarded as a power series $F(X,Y)\in A[[X,Y]]$. This grading is chosen for compatibility with the appearance of formal groups in algebraic topology. Specifically, with our chosen grading, Quillen's isomorphism \cite{MR0253350} of $L^{\mathbb{Z}}$ with the coefficient ring $MU_*$ of complex bordism is an isomorphism of graded rings.
\item Occasionally (e.g. in Theorem \ref{quadratic case} and in Theorem \ref{fourth root of -18}) we refer to a tensor product of {\em infinitely many} commutative algebras. In every case, these are tensor products of commutative {\em graded} algebras, trivial in negative degrees, and in which only finitely many tensor factors have nonzero homogeneous elements in any given degree. The upshot is that there are no special subtleties or surprises in these infinitary tensor products: these infinitary tensor products behave just as ordinary finite tensor products behave.
\item
If $n$ is an integer, $n>1$, then we let $\nu(n)$ stand for the integer $1$ if $n$ is not a prime power, and we let $\nu(n)=p$ if $n$ is a power of a prime number $p$. The function $\nu$ is defined only on the set of integers greater than $1$; in particular, $\nu(1)$ is not defined. (This definition, and the way we use it in the context of formal modules, is not new; for example, see \cite{MR0384707} and \cite{MR506881}). 
\end{itemize}
\end{conventions}

\subsection{Review of some known facts}

Proposition \ref{drinfeld presentation} appears as Proposition 1.1 in \cite{MR0384707}.
\begin{prop}\label{drinfeld presentation} {\bf (Drinfeld.)} Let $A$ be a commutative ring, and let $n$ be an integer. Let $\nu(n)$ be defined as in Convention \ref{nu convention}. 
Let $D^A$ denote the homogeneous ideal in $L^A$ generated by all products of elements $xy$ with $x,y\in L^A$ each homogeneous of positive degree.
Let $\overline{L}^A$ denote the quotient ring $L^A/D^A$.  
The ring $\overline{L}^A$ is graded, so we may consider its degree $n$ summand $\overline{L}^A_n$ for various integers $n$.
If $n\geq 2$, then $\overline{L}^A_{2n-2}$ is isomorphic
to the $A$-module generated by symbols $d$
and $\{ c_a: a\in A\}$, that is, one generator $c_a$ for each element
$a$ of $A$ along with one additional generator $d$,
modulo the relations:
\begin{align}
\label{hazewinkel relation 20} d(a^n-a) &= \nu(n)c_a \mbox{\ \ for\ all\ } a\in A \\
\label{hazewinkel relation 21} c_{a+b}-c_a-c_b &= d\frac{(a+b)^n - a^n - b^n}{\nu(n)} \mbox{\ \ for\ all\ } a,b\in A \\
\label{hazewinkel relation 22} ac_b + b^nc_a &= c_{ab} \mbox{\ \ for\ all\ } a,b\in A .\end{align}
\end{prop} 
We will refer to the presentation described in Proposition \ref{drinfeld presentation} as {\em Drinfeld's presentation for $\overline{L}^A_{2n-2}$.} 

Our gradings are twice those found in \cite{MR0384707}, for the reason explained in Conventions \ref{nu convention}.

One fairly easy application of Proposition \ref{drinfeld presentation} is Proposition \ref{lazard ring for algebras over rationals}, which is proven in \cite{MR0384707}.
\begin{prop}\label{lazard ring for algebras over rationals}
Let $A$ be a commutative $\mathbb{Q}$-algebra. 
Then the classifying ring $L^A$ of formal 
$A$-modules is isomorphic, as a graded $A$-algebra, to
$A[x_1, x_2, \dots ]$. For each $n\geq 1$, the polynomial generator $x_n$ is in degree $2n$, and it corresponds to the generator $d$ of 
$A \cong \overline{L}^A_{2n}$ in the Drinfeld presentation
for $\overline{L}^A_{2n}$.
\end{prop}

Theorem \ref{lazard ring localization iso} appears in Hazewinkel's excellent book \cite{MR506881} (see the proof of Theorem 21.3.5 there). A more fully explained proof also appears in the preprint \cite{cmah1}.
\begin{theorem}\label{lazard ring localization iso}
Let $A$ be a commutative ring and let $S$ be a multiplicatively closed subset of $A$. 
Then the homomorphism of graded rings $L^A[S^{-1}] \rightarrow L^{A[S^{-1}]}$ is an isomorphism. 
\end{theorem}

Definition-Proposition \ref{def of fundamental functional} first appeared, with proof, in a version of the preprint \cite{cmah1}. For the sake of the self-containedness of the present paper, we have elected to remove the proof of Definition-Proposition \ref{def of fundamental functional} from \cite{cmah1}, and to include it in the present paper instead. 
\begin{definition-proposition}\label{def of fundamental functional}
Let $n$ be a positive integer, let $\nu(n)$ be as defined in Conventions \ref{nu convention}, and let $A$ be a commutative ring which is $\nu(n)$-torsion-free. Recall from Proposition \ref{drinfeld presentation} 
that $\overline{L}^A_{2n-2}$ is 
generated, as an $A$-module, by elements $d$ and $\{ c_a\}_{a\in A}$, subject to the relations \eqref{hazewinkel relation 20}, \eqref{hazewinkel relation 21}, and \eqref{hazewinkel relation 22}.

Let $M^A_{2n-2}$ denote the $A$-module generated by elements $d$ and $\{ c_a\}_{a\in A}$, subject only to the relations \eqref{hazewinkel relation 20}.
Let $q_{2n-2}: M^A_{2n-2} \rightarrow \overline{L}^A_{2n-2}$ denote the $A$-module quotient map which enforces the additional relations \eqref{hazewinkel relation 21} and \eqref{hazewinkel relation 22}.

By the {\em degree $2n-2$ fundamental functional of $A$}, we mean the $A$-module homomorphism
\begin{align*} \sigma_{2n-2}: \overline{L}^A_{2n-2} &\rightarrow A\end{align*}
given by 
\begin{align*} 
 \sigma_{2n-2}(d) &= \nu(n),\\
 \sigma_{2n-2}(c_a) &= a^n-a.
\end{align*}

If $n>1$, then the kernel of the composite map $\sigma_{2n-2}\circ q_{2n-2}: M^A_{2n-2}\rightarrow A$
is exactly the set of $\nu(n)$-torsion elements of $M^A_{2n-2}$. Furthermore, the kernels of $\sigma_{2n-2}$ and of $q_{2n-2}$ are 
annihilated by multiplication by $\nu(n)$. 
If $n$ is not a prime power, then $\sigma_{2n-2}$ and $q_{2n-2}$ are isomorphisms of $A$-modules.
\end{definition-proposition}
\begin{proof}
By the definition of $M^A_{2n-2}$, the top row in the commutative diagram of $A$-modules
\begin{equation*}\label{diagram 348} \xymatrix{ 
 \coprod_{a\in A} A\{ r_a\} \ar[r]^(.4){\delta_0} 
  & A\{ d\} \oplus \coprod_{a\in A} A\{ c_a\} \ar[r]^(.7){\delta_{-1}} \ar[rd]_{\tilde{\sigma}_{2n-2}}
  & M^A_{2n-2} \ar[r]\ar[d]^{\sigma_{2n-2}\circ q_{2n-2}}
  & 0 \\
 & & A & }\end{equation*}
is exact, where $\tilde{\sigma}_{2n-2}$ is the map given by
$\tilde{\sigma}_{2n-2}(d) = \nu(n)$ and $\tilde{\sigma}_{2n-2}(c_a) = a^n-a$, where
$\delta_0$ is the map given by 
$\delta_0(r_a) = (a^n-a)d-\nu(n)c_a$, and where $\delta_{-1}$ sends $d$ to $d$ and $c_a$ to $c_a$.
Suppose that $x = \beta d + \sum_{a\in A} \alpha_a c_a$ is in the kernel of $\tilde{\sigma}_{2n-2}$, where
$\beta\in A$ and $\alpha_a\in A$ for all $a\in A$.
Then:
\begin{align*}
 0 
  &= \tilde{\sigma}_{2n-2}(x) \\
  &= \beta\nu(n) + \sum_{a\in A} \alpha_a (a^n - a),
\end{align*}
so $\beta\nu(n) = -\sum_{a\in A} \alpha_a (a^n - a)$.
Continuing, we have
\begin{align*}
 \delta_0\left(\sum_{a\in A} -\alpha_a r_a\right) 
  &= \left( \sum_{a\in A} -\alpha_a (a^n - a)\right) d + \sum_{a\in A} \alpha_a\nu(n) c_a \\
  &= \beta \nu(n) d + \sum_{a\in A} \alpha_a\nu(n) c_a \\
  &= \nu(n) x,\end{align*}
so $\nu(n)x\in \im \delta_0$.
Hence, every element in the kernel of $\sigma_{2n-2}\circ q_{2n-2}$ is killed by $\nu(n)$.

Conversely, since the codomain of $\sigma_{2n-2}\circ q_{2n-2}$ is $A$, and since $A$ is $\nu(n)$-torsion-free, every element in $M^A_{2n-2}$ killed by $\nu(n)$ is also in the kernel of $\sigma_{2n-2}\circ q_{2n-2}$.

By its construction, $q_{2n-2}$ is a surjection, and so we have a short exact sequence
\[ 0 \rightarrow \ker q_{2n-2} \rightarrow \ker \sigma_{2n-2}\circ  q_{2n-2} \rightarrow \ker \sigma_{2n-2} \rightarrow 0.\]
We have just shown that every element in $\ker \sigma_{2n-2}\circ q_{2n-2}$ is killed by multiplication by $\nu(n)$. Since $\ker\sigma_{2n-2}$ and $\ker q_{2n-2}$ are subquotients of $\ker \sigma_{2n-2}\circ q_{2n-2}$, they too are killed by multiplication
by $\nu(n)$, as claimed. 

Now suppose that $n$ is not a prime power. We have just shown that the kernel of $\sigma_{2n-2}$ is killed by multiplication by $\nu(n)$, so 
$\sigma_{2n-2}$ is injective. Furthermore, $\nu(n) = 1$ implies that
$\sigma_{2n-2}(d) = 1\in A$, so $\sigma_{2n-2}$ is surjective. Hence, $\sigma_{2n-2}$ is an isomorphism.
One also checks easily that, since $\nu(n) = 1$, the relations \eqref{hazewinkel relation 21} and \eqref{hazewinkel relation 22}
can be derived from the relation \eqref{hazewinkel relation 20}, so $q_{2n-2}$ is also an isomorphism.
\end{proof}

Below, in Theorem \ref{fund func is inj}, we will show that the fundamental functional is injective whenever $A$ is $\nu(n)$-torsion-free. 
It is the injectivity of the fundamental functional which makes the computation of $L^A$ possible using the methods of this paper. Barring future developments, the entire apparatus of $U$-homology, which is defined and developed starting in the next section, is only a tool for showing that the fundamental functional is injective.

\section{$U$-homology}\label{u-homology section}

\subsection{$U$-homology as the obstruction to $L^A$ being a polynomial algebra}\label{u-homology general properties subsection}

In this subsection we introduce ``$U$-homology,'' an invariant of commutative rings.
In Proposition \ref{u-homology detects injectivity of fundamental functional} 
and in Proposition \ref{fundamental functional main properties ii} 
we demonstrate the main properties of $U$-homology:
\begin{itemize}
\item in dimension $1$, it is the obstruction to injectivity of the fundamental functional,
\item in dimension $0$, it is the obstruction to surjectivity of the fundamental functional,
\item and the vanishing of $U^A_0(n)$ and $U^A_1(n)$ for all $n$ is equivalent to $L^A$ being isomorphic to a polynomial algebra by a certain fundamental comparison map.
\end{itemize}

\begin{definition}\label{def of Fn(A)}
When $A$ is a commutative ring and $n>1$ an integer, let $F_n(A)$ denote the $A/\nu(n)$-module with one generator 
$c_a$ for each element $a\in A$, and subject to the relation $c_0 = 0$ and the relation
$c_{a+b} = c_a + c_b$ for each $a,b\in A$.
\end{definition}
To be clear, in Definition \ref{def of Fn(A)} and throughout this paper, $A/\nu(n)$ denotes the quotient of $A$ by the ideal generated by the integer $\nu(n)$.


\begin{definition-proposition}\label{def of u-homology}
Suppose $A$ is a commutative ring and $n>1$ an integer which is a power of a prime number (which is necessarily $\nu(n)$).
Given an element $a\in A$, we will write $\overline{a}$ for the reduction of $a$ modulo $\nu(n)$.
Let
$\mathcal{U}^A(n)_{\bullet}$ denote the 
simplicial $A/\nu(n)$-module given as follows:
\begin{itemize}
\item $\mathcal{U}^A(n)_{0} = A/\nu(n)$.
\item $\mathcal{U}^A(n)_{m} = F_n(A)^{\otimes_{A/\nu(n)}\, m}$,
i.e., the $m$-fold tensor product, over $A/\nu(n)$, of $F_n(A)$ with itself. This definition holds for $m\geq 0$, and is consistent with the case $m=0$ given above.
\item The face map $d_0: \mathcal{U}^A(n)_1 = F_n(A) \rightarrow A/\nu(n) = \mathcal{U}^A(n)_0$ is given by letting $d_0(c_a) = \overline{a}$.
\item The face map $d_1: \mathcal{U}^A(n)_1 = F_n(A) \rightarrow A/\nu(n) = \mathcal{U}^A(n)_0$ is given by letting $d_1(c_a) = \overline{a}^n$.
\item If $m\geq 1$, the face map \begin{align*} d_0: \mathcal{U}^A(n)_{m+1} = F_n(A)^{\otimes_{A/\nu(n)}\, m+1} \rightarrow  F_n(A)^{\otimes_{A/\nu(n)}\, m} = \mathcal{U}^A(n)_{m}\end{align*} is given by $d_0(c_{a_1} \otimes \dots \otimes c_{a_{m+1}}) = \overline{a}_1(c_{a_2}\otimes \dots \otimes c_{a_{m+1}})$.
\item If $m\geq 1$ and $1\leq i\leq m$, the face map \begin{align*} d_{i}: \mathcal{U}^A(n)_{m+1} = F_n(A)^{\otimes_{A/\nu(n)}\, m+1} \rightarrow F_n(A)^{\otimes_{A/\nu(n)}\, m} = \mathcal{U}^A(n)_{m}\end{align*} is given by \begin{align*} 
d_{i}(c_{a_1} \otimes \dots \otimes c_{a_{m+1}}) &= c_{a_1} \otimes \dots \otimes c_{a_{i-1}}\otimes c_{a_ia_{i+1}}\otimes c_{a_{i+2}} \otimes \dots \otimes c_{a_{m+1}}.\end{align*}
\item If $m\geq 1$, the face map \begin{align*}
d_{m+1}: \mathcal{U}^A(n)_{m+1} = F_n(A)^{\otimes_{A/\nu(n)}\, m+1} &\rightarrow F_n(A)^{\otimes_{A/\nu(n)}\, m} = \mathcal{U}^A(n)_{m}\end{align*} is given\footnote{The formulas for $d_{m+1}$ and for $d_1$ each involve an $n$th power of $\overline{a}$. Taking an $n$th power is, in general, not a group homomorphism. However, the simplicial module we are defining is a simplicial $A/\nu(n)$-module. When $n$ is not a prime power, $A/\nu(n)$ is the zero ring. Consequently,  it is only in the case that $n$ is a prime power that one needs to be sure that the face and degeneracy maps are indeed well-defined module homomorphisms. If $n$ is instead a power of some prime $p$, then $A/\nu(n)$ is a commutative ring of characteristic $p$, so the $n$th power map indeed commutes with addition.} by $d_{m+1}(c_{a_1}\otimes \dots \otimes c_{a_{m+1}}) = \overline{a}_{m+1}^n(c_{a_1} \otimes \dots \otimes c_{a_m})$. 
\item The degeneracy map $s_0: A/\nu(n) = \mathcal{U}^A(n)_0 \rightarrow \mathcal{U}^A(n)_1 = F_n(A)$ is given by $s_0(b) = bc_1$.
\item If $m\geq 1$ and $0\leq i\leq m$, the degeneracy map \begin{align*} s_{i}: \mathcal{U}^A(n)_{m} = F_n(A)^{\otimes_{A/\nu(n)}\, m} &\rightarrow  F_n(A)^{\otimes_{A/\nu(n)}\, m+1} = \mathcal{U}^A(n)_{m+1}\end{align*} is given by 
\begin{align}\label{eq 2096} s_{i}(c_{a_1}\otimes \dots \otimes c_{a_m}) = c_{a_1}\otimes \dots \otimes c_{a_i} \otimes c_1\otimes c_{a_{i+1}} \otimes \dots \otimes c_{a_m}.\end{align}
In the case $i=0$, the right-hand side of formula \eqref{eq 2096} is to be understood as $c_1\otimes c_{a_1} \otimes \dots \otimes c_{a_m}$.
\end{itemize}

Let $U^A(n)_{\bullet}$ denote the Moore/alternating sum chain complex of $\mathcal{U}^A(n)_{\bullet}$, i.e., the chain complex whose $i$-chain group is $\mathcal{U}^A(n)_i$ and whose boundary map $\mathcal{U}^A(n)_{i}\rightarrow \mathcal{U}^A(n)_{i-1}$ is $\sum_{j=0}^i (-1)^jd_j$.
For each $i\geq 0$, let $U^A_i(n)$ denote the $i$th homology group
$H_i(U^A(n)_{\bullet})$ of the chain complex of $A/\nu(n)$-modules $U^A(n)_{\bullet}$. We will call these
homology groups {\em $U$-homology.}


If $n>1$ is an integer which is {\em not} a prime power, then we define $\mathcal{U}^A(n)_{\bullet}$, $U^A(n)_{\bullet}$, and the $U$-homology groups $U^A_i(n)$ to all be zero.
\end{definition-proposition}
\begin{proof}
One ought to show that $\mathcal{U}^A(n)_{\bullet}$ is actually a simplicial $A/\nu(n)$-module, i.e., that the simplicial identities are satisfied. 
This is routine and left as an exercise for the interested reader, who has at least two reasonable approaches: 
\begin{itemize}
\item explicitly write out the simplicial identities and verify that they hold, or 
\item observe that $\mathcal{U}^A(n)_{\bullet}$ is levelwise isomorphic to the Hochschild bar construction of $A/\nu(n)$ with coefficients in a certain bimodule, whose definition is given below in Definition \ref{def of Phi bimodule}, and the levelwise isomorphisms commute with the face and degeneracy maps. Since the face and degeneracy maps in the Hochschild bar construction satisfy the simplicial identities, so must the face and degeneracy maps in $\mathcal{U}^A(n)_{\bullet}$. This argument appears more explicitly below, in the proof of Theorem \ref{u-homology is hh}.
\end{itemize}
\end{proof}
For example, in low degrees, the chain complex $U^A(n)_{\bullet}$ is
 \begin{equation}\label{u-homology cplx} \dots \rightarrow F_n(A)\otimes_{A/\nu(n)} F_n(A)
  \stackrel{\delta_1}{\longrightarrow} 
 F_n(A) \stackrel{\delta_0}{\longrightarrow} A/\nu(n) \rightarrow 0\end{equation}
with $\delta_0$ and $\delta_1$ defined by:
\begin{align*}
 \delta_0(c_{a}) &= \overline{a}-\overline{a}^n,\\
 \delta_1(c_a\otimes c_b) &= \overline{a}c_b - c_{ab} + \overline{b}^nc_a.
\end{align*}
In the rest of this paper, we do not need any of the $U$-homology groups except $U_0^A(n)$ and $U_1^A(n)$, so in fact the complex \eqref{u-homology cplx} suffices for our needs, and the rest of the simplicial module $\mathcal{U}^A(n)_{\bullet}$ is not really needed for anything that follows in this paper. 

\begin{definition-proposition}\label{Pn def-prop}
Let $n>1$ be an integer.
Let $P_{2n-2}^A$ denote the cokernel of the $A$-module homomorphism
$A \rightarrow \overline{L}^A_{2n-2}$
sending $1$ to $d$ (see Proposition \ref{drinfeld presentation} for the
element $d$). Clearly $P_{2n-2}^A$ is functorial in the choice of 
commutative ring $A$.

For any commutative ring $A$, the natural map of $A$-modules
\[ P_{2n-2}^A \rightarrow P_{2n-2}^{A/\nu(n)}\]
is an isomorphism.
\end{definition-proposition}
\begin{proof}
After reducing modulo $d$, the Drinfeld relations \eqref{hazewinkel relation 20}, \eqref{hazewinkel relation 21}, and \eqref{hazewinkel relation 22} become
\begin{align}
\nonumber 0 &= \nu(n) c_a,\\
\nonumber c_{a+b} &= c_a + c_b, \\
\label{haz rel 22a} c_{ab} &= ac_b + b^nc_a.\end{align}
In particular, filling in $1$ for $a$ and $b$ in \eqref{haz rel 22a} gives us that $c_1 = 0$. Hence, $c_{\nu(n)} = \nu(n)c_1 = 0$, and
so $c_{\nu(n)a} = \nu(n)c_a + a^nc_{\nu(n)} = 0$.
Hence, $P_{2n-2}^A$ is isomorphic, as an $A$-module, to the $A/\nu(n)$-module
with one generator $c_a$ for each $a\in A/\nu(n)$, subject to the
modulo $d$ Drinfeld relations above; this is exactly the modulo $d$ 
Drinfeld presentation
for $\overline{L}_{2n-2}^{A/\nu(n)}$, i.e., a presentation for $P_{2n-2}^{A/\nu(n)}$.
\end{proof}

\begin{prop}\label{u-homology detects injectivity of fundamental functional}
Let $A$ be a commutative ring and let $n>1$ be a power of a prime.
Suppose that $A$ is $\nu(n)$-torsion-free. 
Then $U^A_1(n)\cong 0$ if and only if the fundamental functional
$\sigma_{2n-2}: \overline{L}^A_{2n-2} \rightarrow A$
is injective.
Furthermore, $U^A_0(n)$ is isomorphic to the cokernel of
$\sigma_{2n-2}$.
\end{prop}
\begin{proof}
We will write $\delta_1,\delta_0$ for the differentials in the chain complex $U^A(n)_{\bullet}$, defined in Definition-Proposition \ref{def of u-homology} and more explicitly in \eqref{u-homology cplx}. 
Clearly (see the proof of Definition-Proposition \ref{Pn def-prop})
the cokernel of $\delta_1$ is $P^A_{2n-2}$, since $\coker \delta_1$ and $P^A_{2n-2}$ are $A/\nu(n)$-modules with the same set
of generators as one another, and the same set of relations as one another. We have a commutative diagram of $A$-modules with exact rows
\begin{equation}\label{comm diag 09593}\xymatrix{
  & A \ar[d]^{\id} \ar[r]^(.4){d} & \overline{L}^A_{2n-2}\ar[d]^{\sigma_{2n-2}}\ar[r] & P^A_{2n-2}\ar[r] \ar[d]^{\tilde{\delta}_0} & 0 \ar[d] \\
 0 \ar[r] & A \ar[r]^{\nu(n)} & A \ar[r] & A/\nu(n) \ar[r] & 0 ,}\end{equation}
where $\tilde{\delta}_0$ is the $A/\nu(n)$-module map sending each $c_a$ to $a^n-a$.
We furthermore have the commutative diagram with exact rows and columns
\[\xymatrix{
 & 0\ar[d] & 0\ar[d] & \\
0 \ar[r] & \im\delta_1 \ar[r]^{\id}\ar[d] &  \im\delta_1 \ar[r]\ar[d] & 0 \ar[d]\\
0 \ar[r] & \ker \delta_0 \ar[r] \ar[d] & F_n(A) \ar[r]^{-\delta_0}\ar[d] & A/\nu(n) \ar[d]^{\id} \\
0 \ar[r] & U_1^A(n) \ar[r] \ar[d] & P^A_{2n-2} \ar[r]^{\tilde{\delta}_0}\ar[d] & A/\nu(n) \\
 & 0 & 0, & 
}\]
so vanishing of $U^A_1(n)$ is equivalent to injectivity of $\tilde{\delta}_0$.
The ``four lemma'' from homological algebra, applied to diagram \eqref{comm diag 09593}, then tells us that 
injectivity of $\tilde{\delta}_0$ is equivalent to injectivity of $\sigma_{2n-2}$.

For the claim about $U^A_0(n)$: 
we have the commutative diagram of $A$-modules with exact rows
\begin{equation}\label{comm diag f049094}\xymatrix{
 A\{d\} \oplus \coprod_{a\in A} A\{ c_a\}\ar[d]^{\pi^{\prime}} \ar[r]^(.65){-\tilde{\sigma}_{2n-2}} & A \ar[d]^{\pi} \ar[r] & \coker \sigma_{2n-2} \ar[r]\ar[d] & 0 \ar[d] \\
 F_n(A) \ar[r]^(.6){\delta_0} & A/\nu(n) \ar[r] & U_0^A(n) \ar[r] & 0 
}\end{equation}
where:
\begin{itemize}
\item $\pi^{\prime}(d) = 0$, 
\item $\pi^{\prime}(\alpha c_a) = \overline{\alpha}c_a$ for all $\alpha\in A$ (using the notation of Definition-Proposition \ref{def of u-homology}), 
\item and $\pi$ is the modulo $\nu(n)$ reduction map. 
\end{itemize}
The map $\tilde{\sigma}_{2n-2}$ is the composite of the Drinfeld presentation $A\{d\} \oplus \coprod_{a\in A} A\{ c_a\} \twoheadrightarrow \overline{L}_{2n-2}^A$ with the fundamental functional $\sigma_{2n-2}: \overline{L}_{2n-2}^A\rightarrow A$. 
Hence, the bottom row in diagram \eqref{comm diag f049094} is the reduction, modulo $\nu(n)$ (and also modulo $d\in A\{d\}\oplus \coprod_{a\in A} A\{c_a\}$), of the top row. In particular,
 $U^A_0(n)$ is the reduction modulo $\nu(n)$
of $\coker \sigma_{2n-2}$. But $\nu(n)$ is already zero in $\coker \sigma_{2n-2}$
since $\sigma_{2n-2}(d) = \nu(n)$, and consequently $\coker \sigma_{2n-2} \cong U_0^A(n)$.
\end{proof}

This is an opportune time to introduce the symmetric algebras and the Rees algebras. Both are classical constructions. 
We will need graded versions as well, which are slightly less classical:
\begin{definition}\label{def of susp symm and susp rees}
Let $A$ be a commutative ring, $I$ an ideal of $A$.
\begin{itemize}
\item By the {\em Rees algebra of $I$}, written $\Rees_A(I)$, we mean the commutative $A$-algebra $\coprod_{n\geq 0} I^n\{ t^n\} \subseteq A[t]$.
\item Let $j$ be an integer. By the {\em $j$-suspended Rees algebra of $I$}, written $\Rees^j_A(I)$, we mean the commutative graded $A$-algebra whose underlying commutative $A$-algebra is $\Rees_A(I)$, and which is equipped with the grading in which the summand $I^n\{ t^n\}$ is in grading degree $jn$.
\end{itemize}
Now, more generally, let $A$ be a commutative ring and let $M$ be an $A$-module.
\begin{itemize}
\item By the {\em symmetric algebra of $M$}, written $\Symm_A(M)$, we mean the commutative $A$-algebra $\coprod_{n\geq 0} (M^{\otimes_A n})_{\Sigma_n}$,
where $(M^{\otimes_A n})_{\Sigma_n}$ is the module of coinvariants under the action of the symmetric group $\Sigma_n$ on $M^{\otimes_A n}$ given by 
permuting the tensor factors.
\item Let $j$ be an integer. By the {\em $j$-suspended symmetric algebra of $M$}, written $\Symm^j_A(M)$, we mean the commutative graded $A$-algebra whose underlying commutative $A$-algebra is $\Symm_A(M)$, and which is equipped with the grading in which the summand $(M^{\otimes_A n})_{\Sigma_n}$ is in degree $jn$.
\end{itemize}
\end{definition}
In this paper, $\Rees^j_A(I)$ only occurs in cases where $j$ is even, so $\Rees^j_A(I)$ is both graded-commutative and commutative. That is, as explained in Convention \ref{nu convention}, we do not have to deal with the Koszul sign convention.

\begin{definition}\label{def of fundamental comparison 1}
Let $A$ be a commutative ring.
We will say that $A$ {\em satisfies the fundamental comparison condition} if the $A$-module $\overline{L}^A_{2n-2}$ is projective for all integers $n>1$.
\end{definition}
If $A$ satisfies the fundamental comparison condition, then the projection $A$-module map $L^A_{2n-2}\rightarrow \overline{L}^A_{2n-2}$ splits for all
integers $n>1$. Choose such a splitting $A$-module map $i_{2n-2}: \overline{L}^A_{2n-2}\rightarrow L^A_{2n-2}$ for each integer $n>1$, and let
$i: \coprod_{n>1} \overline{L}^A_{2n-2}\rightarrow L^A$
be the $A$-module coproduct of the maps $i_{2n-2}$. Then the adjunction between $\Symm_A$ and the forgetful functor from commutative $A$-algebras
to $A$-modules yields
a choice of commutative graded $A$-algebra homomorphism $i^{\sharp}: \Symm_A\left( \coprod_{n>1} \overline{L}^A_{2n-2}\right) \rightarrow L^A$.
\begin{definition}\label{def of fundamental comparison 2} Let $A$ be a torsion-free commutative ring satisfying the fundamental comparison condition. 
By the {\em fundamental comparison triangle for $A$} we mean the diagram of commutative graded $A$-algebra homomorphisms
\begin{equation}\label{fct} \xymatrix{ &  \Symm_A\left( \coprod_{n > 1} \overline{L}^A_{2n-2} \right) \ar[rd]^{s}\ar[ld]_{i^{\sharp}}  & & \\
L^A & & \Symm_A\left( \coprod_{n > 1} A\right) \ar[r]^{\cong} & A[x_1, x_2, \dots ] }\end{equation}
where $s$ is $\Symm_A\left(\coprod_{n > 1}\sigma_{2n-2}\right)$, the symmetric algebra functor applied to the coproduct (in the category of graded $A$-modules) of the fundamental functionals $\sigma_{2n-2}$ for all $n$. The $A$-algebra $A[x_1, x_2, \dots ]$ is graded, with $x_i$ in degree $2i$.
\end{definition}
Note that we need $A$ to satisfy the fundamental comparison condition in order to define the fundamental comparison triangle.
\begin{definition}\label{def of fundamental comparison 3} Let $A$ be a torsion-free commutative ring satisfying the fundamental comparison condition. Choose a splitting $A$-module map for $L^A_{2n-2} \rightarrow \overline{L}^A_{2n-2}$ for each integer $n>1$. 
We will say that {\em $L^A$ is polynomial by the fundamental comparison (with respect to the given family of splitting maps)} if each $A$-algebra homomorphism
in the fundamental comparison triangle is an isomorphism.
\end{definition}
\begin{remark}\label{non-naturality of fct}
The fundamental comparison triangle is, at least {\em a priori}, not natural in the choice of $A$, since it involves making choices of the splitting 
maps $\{ i_{2n-2}\}$, and when one has a homomorphism from one split short exact sequence to another, there is an (often nontrivial) obstruction 
to the existence of a {\em compatible} splitting of the two short exact sequences. See \cite{compatiblesplittings}, for example.
\end{remark}

\begin{remark}\label{symm and monos and epis} Although $\Symm_A$ preserves epimorphisms, it typically does not preserve monomorphisms; see section 6.2 of chapter III of \cite{MR1727844}. \end{remark}

\begin{prop}\label{fundamental functional main properties ii}
Suppose that $A$ is a torsion-free commutative ring satisfying the fundamental comparison condition. 
Then the following are equivalent:
\begin{enumerate}
\item The groups $U^A_1(n)$ and $U^A_0(n)$ are both trivial for all integers $n>1$.
\item $L^A$ is polynomial by the fundamental comparison, with respect to every family of splitting maps.
\item $L^A$ is polynomial by the fundamental comparison, with respect to some family of splitting maps.
\end{enumerate}
\end{prop}
\begin{proof}
From Theorem \ref{lazard ring localization iso} we know that the natural commutative graded $\mathbb{Q}\otimes_{\mathbb{Z}}A$-algebra homomorphism
$\mathbb{Q}\otimes_{\mathbb{Z}} L^A \rightarrow L^{\mathbb{Q}\otimes_{\mathbb{Z}} A}$
is an isomorphism. 
It is automatic that $\mathbb{Q} \otimes_{\mathbb{Z}}A$ satisfies the fundamental comparison condition, since $\mathbb{Q} \otimes_{\mathbb{Z}}A$ is torsion-free and consequently Proposition \ref{lazard ring for algebras over rationals} ensures that the $\mathbb{Q}\otimes_{\mathbb{Z}} A$-module $\overline{L}^{\mathbb{Q}\otimes_{\mathbb{Z}} A}_{2n-2}$ is free, hence projective. 
After choosing a family of splitting maps, we can fit the fundamental comparison triangle for $A$ together with the fundamental comparison triangle for $\mathbb{Q}\otimes_{\mathbb{Z}} A$
to get the diagram
\begin{equation}\label{diag 09902} \xymatrix{
 &  \Symm_A\left( \coprod_{n> 1} \overline{L}^A_{2n-2} \right) \ar[rd]^{s}\ar[ld]_{c} \ar[dd]^{\ell^{\prime}} &  \\
L^A\ar[dd]^{\ell^{\prime\prime}} & & \Symm_A\left( \coprod_{n> 1} A\right)\ar[dd]^{\ell} \\
 &  \Symm_{\mathbb{Q}\otimes_{\mathbb{Z}}A}\left( \coprod_{n> 1} \overline{L}^{\mathbb{Q}\otimes_{\mathbb{Z}}A}_{2n-2} \right) \ar[rd]^{\overline{s}}\ar[ld]_{\overline{c}}  &  \\
L^{\mathbb{Q}\otimes_{\mathbb{Z}}A} & & \Symm_{\mathbb{Q}\otimes_{\mathbb{Z}}A}\left( \coprod_{n> 1} \mathbb{Q}\otimes_{\mathbb{Z}}A\right)
}\end{equation}
In light of Remark \ref{non-naturality of fct}, we ought to explain why we have a map of fundamental comparison triangles in this situation.
The reason is that, when we choose a splitting $i_{2n-2}$ of the projection $\pi_{2n-2}: L^A_{2n-2} \rightarrow \overline{L}^A_{2n-2}$, we can tensor that splitting map
over $A$ with $\mathbb{Q}\otimes_{\mathbb{Z}} A$ to get a splitting map $\mathbb{Q}\otimes_{\mathbb{Z}} L^A_{2n-2} \rightarrow \mathbb{Q}\otimes_{\mathbb{Z}} \overline{L}^A_{2n-2}$, and hence, using Theorem \ref{lazard ring localization iso}, a commutative diagram of $A$-modules
\[ \xymatrix{
 \overline{L}^A_{2n-2}\ar@/^2ex/[rrrr]^{\id}\ar[d]\ar[rr]_{i_{2n-2}} && L^A_{2n-2}\ar[d]\ar[rr]_{\pi_{2n-2}} && \overline{L}^A_{2n-2} \ar[d] \\
 \overline{L}^{\mathbb{Q}\otimes_{\mathbb{Z}}A}_{2n-2}\ar@/_2ex/[rrrr]_{\id}\ar[rr]^{i_{2n-2}} && L^{\mathbb{Q}\otimes_{\mathbb{Z}}A}_{2n-2}\ar[rr]^{\pi_{2n-2}} && \overline{L}^{\mathbb{Q}\otimes_{\mathbb{Z}}A}_{2n-2}, } \]
which is all we need in order to get the commutativity of the left-hand parallelogram in diagram \eqref{diag 09902}. 
\begin{description}
\item[1 implies 2]
Suppose that $U_0^A(n) \cong 0 \cong U_1^A(n)$ for all integers $n>1$.
Then $\sigma_{2n-2}$ is an isomorphism for all $n>1$ by Proposition \ref{u-homology detects injectivity of fundamental functional},
hence the homomorphism $s$ in diagram \eqref{fct} is an isomorphism. 
Since $A$ is torsion-free, the localization map $A \rightarrow \mathbb{Q}\otimes_{\mathbb{Z}} A$ is injective, hence
the map marked $\ell$ in diagram \eqref{diag 09902} is injective\footnote{This {\em isn't} due to $\Symm$ preserving injections, since
in general, $\Symm$ does not preserve injections---see Remark \ref{symm and monos and epis}. Instead this is simply due to the observation
that $\Symm_A(\coprod_{n> 1} A)$ is a polynomial algebra over $A$, and $\Symm_{\mathbb{Q}\otimes_{\mathbb{Z}}A}\left(\mathbb{Q}\otimes_{\mathbb{Z}} \coprod_{n> 1} A\right)$ is a polynomial algebra over $\mathbb{Q}\otimes_{\mathbb{Z}} A$ on the same set of polynomial generators.}. 
Consequently,  $\ell\circ s= \overline{s} \circ \ell^{\prime}$ is injective, so $\ell^{\prime}$ is injective as well.
The map marked $\overline{c}$ in diagram \eqref{diag 09902} is an isomorphism by Proposition \ref{lazard ring for algebras over rationals},
so $\overline{c}\circ\ell^{\prime} = \ell^{\prime\prime}\circ c$ is injective, hence $c$ is injective.
The map $c$ is also surjective since every element in $L^A$ is a sum of products of indecomposables.
Hence, $c$ is an isomorphism. Hence, $c$ and $s$ are both isomorphisms, so $L^A$ is polynomial by the fundamental comparison.
\item[2 implies 3] Trivial.
\item[3 implies 1] If $L^A$ is polynomial by the fundamental comparison, then the map $s$ in diagram \eqref{diag 09902} is an isomorphism.
Since the map of $A$-modules $\coprod_{n > 1} \sigma_{2n-2}: \coprod_{n > 1} \overline{L}^A_{2n-2} \rightarrow \coprod_{n > 1} A$
is the summand consisting of rank $1$ tensors in the map of $A$-modules
$s: \Symm_A\left(\coprod_{n > 1} \overline{L}^A_{2n-2}\right) \rightarrow \Symm_A\left(\coprod_{n > 1} A\right)$,
we then have that $\coprod_{n>1}\sigma_{2n-2}$ is an isomorphism of $A$-modules, and hence that each $\sigma_{2n-2}$ is an isomorphism of
$A$-modules.
Now Proposition \ref{u-homology detects injectivity of fundamental functional} implies that $U_0^A(n)\cong 0 \cong U_1^A(n)$ 
for all integers $n>1$.
\end{description}
\end{proof}


\begin{definition}\label{def of In ideal}
Let $A$ be a commutative ring, and let $n$ be a positive integer. We write $I^A_n$ for the ideal in $A$ generated by $\nu(n)$ and by all elements of the form $a^n-a\in A$. 
\end{definition}
When $n$ is not a prime power, it is easy to see that the ideal $I^A_n$ is the trivial ideal $(1)$. For a prime power $n=p^m$, the ideal $I^A_{p^m}$ has a certain universal property: it is the largest among all ideals in $A$ which are contained in the kernel of every ring homomorphism $A\rightarrow \mathbb{F}_{p^m}$. For that reason, the ideal $I^A_{p^m}$ is called {\em the universal $\mathbb{F}_{p^m}$-point-detecting ideal of $A$} in the preprint \cite{cmah1}. See \cite{cmah1} for further discussion of the relationship between $I^A_{p^m}$, counting $\mathbb{F}_{p^m}$-points, and the zeta-function of $\Spec A$.

\begin{remark}\label{description of ideals I_n}
It is occasionally useful (e.g. in Corollary \ref{number ring iso}, below) to have a smaller presentation of the 
ideal $I^A_{n}$. Suppose that $A$ is generated as a commutative ring by a single element $t\in A$, 
i.e., the ring homomorphism $\mathbb{Z}[t] \rightarrow A$, sending $t$ to $t$, is surjective. 
Then $I^A_{n} = (\nu(n),t^{n}-t)$.
The proof is elementary, and left to the interested reader.
\end{remark}

\begin{theorem}\label{general thm}
Suppose that $A$ is a torsion-free commutative ring, 
and suppose that the fundamental functional $\sigma_{2n-2}$ is injective for all $n>1$.
Suppose that, for each $n>1$, the ideal $I^A_{n}$ is projective when regarded as an $A$-module.
Then the following statements are all true:
\begin{itemize}
\item $A$ satisfies the fundamental comparison condition.
\item Of the fundamental comparison maps 
\begin{align*} s: \Symm_A\left( \coprod_{n > 1} \overline{L}^A_{2n-2}\right)\rightarrow \Symm_A\left( \coprod_{n > 1} A\right)\end{align*}
and \begin{align*} i^{\sharp}: \Symm_A\left( \coprod_{n > 1} \overline{L}^A_{2n-2}\right) &\rightarrow L^A,\end{align*}
the map $s$ is injective and $i^{\sharp}$ is an isomorphism. Consequently,  $L^A$ is isomorphic to a sub-$A$-algebra of a polynomial $A$-algebra.
\item $L^A$ is isomorphic, as a graded $A$-algebra, to the tensor product of the suspended symmetric algebras of the ideals $I^A_2, I^A_3, \dots$ of $A$:
\begin{align}
\label{LA iso 23030} L^A 
  &\cong \bigotimes_{n>1} \Symm_A^{2n-2}\left( I^A_{n}\right) \\
\nonumber  &\cong \Symm_A^2(I^A_2) \otimes_A \Symm_A^4(I^A_3) \otimes_A \Symm_A^6(I^A_4) \otimes_A \dots 
.\end{align}
\item $L^A$ is isomorphic, as a graded $A$-algebra, to the tensor product of the suspended Rees algebras of the ideals $I^A_2, I^A_3, \dots$ of $A$:
\begin{align} 
\label{LA iso 23031} L^A 
  &\cong \bigotimes_{n>1} \Rees_A^{2n-2}\left( I^A_{n}\right) \\
\nonumber  &\cong \Rees_A^2(I^A_2) \otimes_A \Rees_A^4(I^A_3) \otimes_A \Rees_A^6(I^A_4) \otimes_A \dots .\end{align}
\end{itemize}
\end{theorem}
\begin{proof}\leavevmode
\begin{itemize}[leftmargin=*]
\item 
Since $\sigma_{2n-2}$ is assumed injective for all $n$, the $A$-module $\overline{L}^A_{2n-2}$ is 
isomorphic to $\im \sigma_{2n-2} = I^A_{n}$, which is projective for all $n>1$, by assumption.
Hence, $A$ satisfies the fundamental comparison condition.
\item
We claim that the inclusion $I_{n}^A\subseteq A$, regarded as an $A$-module
map, induces a monomorphism 
\begin{align}\label{mono 3204934}\Symm_A(I_{n}^A) &\rightarrow \Symm_A(A)\end{align}
after applying $\Symm_A$. The injectivity of \eqref{mono 3204934} uses the assumption that $I_{n}^A$ is projective, and is not generally true if we relax the projectivity assumption on $I^A_{n}$; see Remark \ref{symm and monos and epis}. 

The argument for injectivity of \eqref{mono 3204934} is general, classical, and quite simple: since $I^A_{n}$ is projective, tensoring the $A$-module monomorphism $I^A_{n} \rightarrow A$ with $I^A_{n}$ yields that the multiplication map $I^A_{n}\otimes_A I^A_{n}\rightarrow I_{n}^A$ is monic. A similar argument together with a straightforward induction yields that the multiplication map $(I^A_{n})^{\otimes_A j}\rightarrow A^{\otimes_A j}\cong A$ is monic for all positive integers $j$. Hence, we have a commutative square of $A$-modules
\begin{equation*}\label{comm sq fk3409a}\xymatrix{
 (I^A_{n})^{\otimes_A j}\ar[d] \ar[r]& A^{\otimes_A j} \ar[d] \\
 \left((I^A_{n})^{\otimes_A j}\right)_{\Sigma_j} \ar[r] & \left(A^{\otimes_A j}\right)_{\Sigma_j}
}\end{equation*}
whose top horizontal map is injective. The right-hand vertical map is an isomorphism, hence the composite from upper-left to lower-right is injective. This tells us that the left-hand vertical map must be injective as well. It is also surjective, since it is projection to the module of coinvariants. Hence, the left-hand vertical map is an isomorphism, hence the bottom horizontal map is injective. The direct sum of the bottom horizontal map, over all nonnegative integers $j$, is precisely the map $\Symm_A(I^A_{n})\rightarrow \Symm_A(A)$. Hence, \eqref{mono 3204934} is injective, as claimed.

The map $s$ is, up to isomorphism, a tensor product of copies of the map \eqref{mono 3204934}, taken across all integers $n>1$. This map is injective by an induction completely analogous to that described in the previous paragraph, again using the projectivity of $I^A_{n}$, and the resulting projectivity of $\Symm_A(I^A_{n})$, as an $A$-module. Hence, $s$ is injective, as claimed.

Now we need to know why the map $i^{\sharp}$ is an isomorphism. 
We have the commutative diagram
\begin{equation}\label{diag 0950945959} \xymatrix{
 & \Symm_A\left( \coprod_{n>1} \overline{L}^A_{2n-2}\right) \ar[ld]_{i^{\sharp}} \ar[rd]^s \ar[d]^{\ell_2}   & \\
 L^A \ar[d]_{\ell_1} & \mathbb{Q}\otimes_{\mathbb{Z}} \Symm_A\left( \coprod_{n>1} \overline{L}^A_{2n-2}\right) \ar[ld]^{\mathbb{Q}\otimes_{\mathbb{Z}} i^{\sharp}} \ar[rd]^{\mathbb{Q}\otimes_{\mathbb{Z}} s}\ar[d]_g & \Symm_A\left( \coprod_{n>1} A\right) \ar[d]^{\ell_3} \\
 \mathbb{Q}\otimes_{\mathbb{Z}} L^A\ar[d]_f & \Symm_{\mathbb{Q}\otimes_{\mathbb{Z}} A}\left( \coprod_{n>1} \overline{L}^{\mathbb{Q}\otimes_{\mathbb{Z}} A}_{2n-2}\right)\ar[ld]^h & \mathbb{Q}\otimes_{\mathbb{Z}} \Symm_A\left( \coprod_{n>1} A\right) \\
 L^{\mathbb{Q}\otimes_{\mathbb{Z}} A} &&
}\end{equation}
in which the vertical maps $\ell_1,\ell_2,\ell_3$
are the canonical localization maps. The map marked $f$ is an isomorphism by Theorem \ref{lazard ring localization iso}, while the map marked $h$ is an isomorphism by Proposition \ref{lazard ring for algebras over rationals}. Since $\mathbb{Z}\rightarrow\mathbb{Q}$ is flat, the map marked $g$ is an isomorphism by general base change properties of the functor $\Symm$; see for example Proposition 5.2(iii) of \cite{MR2418168}. 
Hence the map $\mathbb{Q}\otimes_{\mathbb{Z}} i^{\sharp}$ 
is an isomorphism as well. 
The map $\ell_2$ is injective, since we already showed that \[ \Sym_A\left( \coprod_{n>1} \overline{L}^A_{2n-2}\right) \cong \Sym_A\left( \coprod_{n>1} I^A_{n}\right) \rightarrow \Sym_A\left( \coprod_{n>1} A\right)\] is injective, $\Sym_A\left( \coprod_{n>1} A\right)$ is a free $A$-module, and $A$ is torsion-free.
Hence, $\left( \mathbb{Q}\otimes_{\mathbb{Z}} i^{\sharp}\right) \circ\ell_2 = \ell_1\circ i^{\sharp}$ is injective, and consequently $i^{\sharp}$ is injective.

The map $i^{\sharp}$ is also surjective since every element of $L^A$ is a product
of (lifts of) indecomposables. Hence, $i^{\sharp}$ is an isomorphism, as claimed.
\item The fact that $i^{\sharp}$ is an isomorphism gives us the claimed isomorphism
of $L^A$ with a tensor product of suspended symmetric algebras.
We already showed that, for each $j$, the $j$th symmetric power of $I^A_{n}$ coincides with the $j$th power $(I^A_{n})^j$ of the ideal $I^A_{n}$, i.e., the $j$th summand in the Rees algebra $\Rees_A(I^A_{n})$. Consequently,  the symmetric algebras of the ideals $I^A_n$ coincide with their Rees algebras,
giving us the claimed isomorphism of $L^A$ with a tensor product of suspended
Rees algebras.

The degree $2n-2$ of the suspensions on the right-hand side of \eqref{LA iso 23030} and of \eqref{LA iso 23031} arises because $I_n^A\subseteq A$ is the image of the fundamental functional $\sigma_{2n-2}: \overline{L}^A_{2n-2} \rightarrow A$, and $\overline{L}^A_{2n-2}$ is the $A$-module of degree $2n-2$ indecomposables in the graded $A$-algebra $L^A$.
\end{itemize}
\end{proof}

\begin{corollary} {\bf (Lifting and extensions.)}\label{lifting and extensions}
Suppose that $A$ is a torsion-free commutative ring, and suppose that the fundamental functional $\sigma_{2n-2}$ is injective, and has image a projective module, for all $n>1$.
Then every formal $A$-module $n$-bud extends to a formal $A$-module.
Furthermore, if $R$ is a commutative $A$-algebra and $I$ is an ideal in $R$, then every formal $A$-module over $R/I$ is the reduction modulo $I$ of a formal $A$-module over $R$.
\end{corollary}
\begin{proof}
Let $L^A_{\leq n}$ be the classifying ring of formal $A$-module $n$-buds,
and let $\iota: L^A_{\leq n} \rightarrow L^A$ be the map classifying
the formal $A$-module $n$-bud of the universal formal $A$-module.
Then from Theorem \ref{general thm} we have a commutative square
\[\xymatrix{ 
 L^A_{\leq n} \ar[rr]^{\iota} \ar[d]^{\cong} && L^A \ar[d]^{\cong} \\
 \Symm_A \left( \coprod_{1<m\leq n}  I^A_m\right) \ar[rr]^{\Symm_A(\kappa)} &&
  \Symm_A \left( \coprod_{1<m} I^A_m\right) }\]
where $\kappa$ is the inclusion of the summand
$\kappa: \coprod_{1<m\leq n}  I^A_m\hookrightarrow \coprod_{1<m} I^A_m$.
By the universal property of $\Symm_A$,
every morphism of commutative $A$-algebras $L^A_{\leq n} \rightarrow R$
extends over $\iota$ to a morphism of commutative $A$-algebras 
$L^A\rightarrow R$, hence every formal $A$-module $n$-bud extends
to a formal $A$-module.

Furthermore, by the universal property of $\Symm_A$ and the lifting property of projective modules, every morphism of commutative $A$-algebras $L^A\rightarrow R/I$ lifts to a morphism $L^A\rightarrow R$. 
Hence, every formal $A$-module over $R/I$ is the reduction modulo $I$ of a formal $A$-module over $R$.
\end{proof}

The fundamental comparison works especially well for hereditary rings $A$:
\begin{corollary}\label{fundamental functional main properties iv}
Suppose that $A$ is a torsion-free hereditary commutative ring,
and suppose that the fundamental functional $\sigma_{2n-2}$ is injective for all $n>1$.
Then the conclusions of Theorem \ref{general thm} and of Corollary \ref{lifting and extensions} all hold.
\end{corollary}

\subsection{Calculation of $U$-homology}\label{calculation of U-homology}

\begin{definition}\label{def of Phi bimodule}
Let $p$ be a prime number, and let $B$ be a commutative algebra over the field $\mathbb{F}_p$ with $p$ elements.  If $n$ is a power of $p$, let $\Phi^n(B)$ denote the $B$-bimodule whose underlying right $B$-module is $B$ itself, and whose left $B$ action is twisted by the $n$th power map. That is, if $x\in \Phi^n(B)$ and $b\in B$, then $xb\in \Phi^n(B)$ is the product of $x$ and $b$ in the ring $B$, and $bx\in \Phi^n(B)$ is the product of $b^n$ and $x$ in the ring $B$.
\end{definition}

As mentioned in the proof of Definition-Proposition \ref{def of u-homology}, the chain complex $U^A(n)_{\bullet}$ in Definition-Proposition \ref{def of u-homology} can be constructed as a Hochschild chain complex. We have the following identification of $U$-homology with Hochschild homology with appropriate coefficients:
\begin{theorem}\label{u-homology is hh}
Let $A$ be a commutative ring, and let $n>1$ be an integer. 
\begin{itemize}
\item If $n$ is not a prime power, then the $U$-homology group $U^A_i(n)$ vanishes for all $i$. 
\item If $n$ is a power of a prime number $p$, then write $A/p$ for the ring $A$ reduced modulo the ideal generated by $p$. Then, for all $i\geq 0$, the $U$-homology group $U^A_i(n)$ is isomorphic to the Hochschild homology group $HH_i(A/p; \Phi^n(A/p))$. 
\end{itemize}
\end{theorem}
\begin{proof}
In the case that $n$ is not a prime power, $A/\nu(n)$ is the zero ring, so the chain complex of $A/\nu(n)$-modules $U^A(n)_{\bullet}$ is zero. The nontrivial case is when we instead suppose that $n$ is a power of $p = \nu(n)$. By Definition \ref{def of Fn(A)}, the $A/p$-module $F_n(A)$ is obtained by forgetting the $A$-module structure on $A$, retaining only the abelian group structure, and then taking the free $A/p$-module on this abelian group. In other words, $F_n(A)$ is $A/p\otimes_{\mathbb{Z}}A$. Of course this $A/p$-module is, in turn, isomorphic to $A/p \otimes_{\mathbb{F}_p} A/p$. Consequently,  the $A/p$-module of $m$-simplices $\mathcal{U}^A(n)_{m}$ in $\mathcal{U}^A(n)_{\bullet}$, defined in Definition-Proposition \ref{def of u-homology}, is
\begin{align*}
 \mathcal{U}^A(n)_{m} 
  &= F_n(A)^{\otimes_{A/p}\, m} \\
  &\cong \left( A/p \otimes_{\mathbb{Z}} A\right)^{\otimes_{A/p}\, m} \\
  &\cong \left( A/p \otimes_{\mathbb{F}_p} A/p\right)^{\otimes_{A/p}\, m} \\
  &\cong A/p \otimes_{\mathbb{F}_p} \left( A/p^{\otimes_{\mathbb{F}_p}\, m}\right) 
\end{align*}
via the isomorphism
\begin{align*}
 F_n(A)^{\otimes_{A/p}\, m} 
  &\stackrel{\psi_m}{\longrightarrow} A/p \otimes_{\mathbb{F}_p} \left( A/p^{\otimes_{\mathbb{F}_p}\, m}\right) \\
 a_1c_{b_1} \otimes a_2 c_{b_2} \otimes \dots \otimes a_mc_{b_m}
  &\mapsto a_1a_2 \dots a_m \otimes \overline{b}_1\otimes \dots \otimes \overline{b}_m.
\end{align*}

We claim that the isomorphisms $\psi_0,\psi_1,\dots$ fit together with the face and degeneracy maps to yield an isomorphism of simplicial $A/p$-modules
\begin{equation}\label{simplicial map 23}\xymatrix{
 F_n(A)^{\otimes_{A/p}\,  0} \ar[d]^{\psi_0}_{\cong}
 \ar[r] & 
  F_n(A)^{\otimes_{A/p}\,  1} \ar[d]^{\psi_1}_{\cong} \ar@<1ex>[l]\ar@<-1ex>[l]\ar@<1ex>[r]\ar@<-1ex>[r] &
  F_n(A)^{\otimes_{A/p}\,  2} \ar[d]^{\psi_2}_{\cong} \ar@<2ex>[l]\ar@<-2ex>[l]\ar[l]
                                 \ar@<2ex>[r]\ar@<-2ex>[r]\ar[r] &
   \dots \ar@<3ex>[l]\ar@<-3ex>[l]\ar@<1ex>[l]\ar@<-1ex>[l] \\
 A/p \otimes_{\mathbb{F}_p} \left( A/p^{\otimes_{\mathbb{F}_p} 0}\right) 
 \ar[r] & 
  A/p \otimes_{\mathbb{F}_p} \left( A/p^{\otimes_{\mathbb{F}_p} 1}\right) \ar@<1ex>[l]\ar@<-1ex>[l]\ar@<1ex>[r]\ar@<-1ex>[r] &
  A/p \otimes_{\mathbb{F}_p} \left( A/p^{\otimes_{\mathbb{F}_p} 2}\right)  \ar@<2ex>[l]\ar@<-2ex>[l]\ar[l]
                                 \ar@<2ex>[r]\ar@<-2ex>[r]\ar[r] &
   \dots \ar@<3ex>[l]\ar@<-3ex>[l]\ar@<1ex>[l]\ar@<-1ex>[l] .  }\end{equation}
The top row of \eqref{simplicial map 23} is the simplicial $A/p$-module $\mathcal{U}^A(n)_{\bullet}$, while the bottom row is the Hochschild bar construction of $A/p$, as a $\mathbb{F}_p$-algebra, with coefficients in the bimodule $\Phi^n(A/p)$. Verifying that the maps $\psi_0,\psi_1,\dots$ indeed define a morphism of simplicial $A/p$-modules is a matter of verifying that they commute with the face and degeneracy maps, which is routine. 

Since each $\psi_i$ is an isomorphism, $\mathcal{U}^A(n)_{\bullet}$ is isomorphic to the Hochschild bar construction of $A/p$ with coefficients in $\Phi^n(A/p)$. Taking Moore complexes and then homology, we get the desired isomorphism $U^A_i(n) \cong HH_i(A/p; \Phi^n(A/p))$ for each $i$.
\end{proof}

The following 2007 theorem of Pirashvili, the main result of \cite{MR2336249}, is now extremely useful:
\begin{theorem} \label{pirashvili thm} {\bf (Pirashvili.)} Let $n$ be a positive power of a prime number $p$, let $B$ be a commutative $\mathbb{F}_p$-algebra, and let $\Phi^n(B)$ be the $B$-bimodule defined in Definition \ref{def of Phi bimodule}. Then $HH_i(B; \Phi^n(B))$ vanishes for all $i>0$.
\end{theorem}

Putting Theorems \ref{u-homology is hh} and \ref{pirashvili thm} together, we have:
\begin{corollary}\label{u-homology and pirashvili cor}
All $U$-homology groups vanish in all positive homological degrees. That is, $U_i^A(n)$ vanishes whenever $A$ is a commutative ring, and $i,n$ are integers with $i>0$ and $n>1$.

Furthermore, the group $U^A_0(n)$ if trivial if $n$ is not a prime power. If $n$ is a power of a prime number $p$, then $U^A_0(n) \cong HH_0(A/p; \Phi^n(A/p)) \cong A/I^A_n$.
\end{corollary}

\subsection{Consequences for the structure of $L^A$}\label{global consequences subsection}

Now we are ready to combine the calculation of $U$-homology in Corollary \ref{u-homology and pirashvili cor} with the relationship between $U$-homology and $L^A$, proven in Proposition \ref{u-homology detects injectivity of fundamental functional}, yielding the following theorem:
\begin{theorem}\label{fund func is inj}
Let $A$ be a commutative ring and let $n>1$ be an integer.
Suppose that $A$ is $\nu(n)$-torsion-free.
Then the fundamental functional $\sigma_{2n-2}: \overline{L}^A_{2n-2} \rightarrow A$
is injective. The image of $\sigma_{2n-2}$ is the ideal $I_{n}^A$ generated by $\nu(n)$ and by all elements of the form
$a^n-a\in A$.
\end{theorem}

Applying Theorem \ref{general thm} now yields:
\begin{corollary}\label{gen thm cor}
Let $A$ be a torsion-free commutative ring. Suppose that, for each integer $n>1$, the ideal $I^A_{n} = (\nu(n),a^n-a\ \forall a\in A)$ of $A$ is projective as an $A$-module. Then $L^A$ is isomorphic to the tensor product of suspended symmetric algebras, and also isomorphic to the tensor product of suspended Rees algebras:
\begin{equation*}
  L^A 
  \cong \bigotimes_{n>1} \Symm_A^{2n-2}\left( I^A_{n}\right) 
  \cong \bigotimes_{n>1} \Rees_A^{2n-2}\left( I^A_{n}\right) .
\end{equation*}
\end{corollary}
In particular, by assuming that the ring $A$ is hereditary, we can force the projectivity hypothesis on the ideals $I_{n}^A$ to be satisfied:
\begin{theorem}\label{main thm in hereditary case}
Let $A$ be a torsion-free commutative ring. 
Let $S$ be a set of prime numbers such that the ring
$A[S^{-1}]$ is hereditary.
(If, for example, $A$ is already hereditary,
then we can let $S$ be the empty set.)

Then the commutative graded ring $L^A$
is, after inverting $S$, isomorphic to a symmetric algebra, and also to a tensor product of (``suspended,'' i.e., graded) Rees algebras:
\begin{align*}
 L^A[S^{-1}] &\cong \left( \Rees^2_{A}(I^A_2) \otimes_{A} \Rees_{A}^4(I^A_3) \otimes_{A} \Rees_{A}^6(I^A_4) \otimes_{A}\Rees_{A}^8(I^A_5) \otimes_{A} \dots\right) [S^{-1}] \\
 &\cong \left(\Symm_{A}\left( \coprod_{n>1} I^{A}_{n}\right)\right)[S^{-1}],
\end{align*}
where $I^A_{n}$, for each positive integer $n$, 
is the ideal of $A$ defined in Definition \ref{def of In ideal} and again in Corollary \ref{gen thm cor}.
\end{theorem}
\begin{proof}
Theorem \ref{lazard ring localization iso} ensures that $L^A[S^{-1}]$ coincides with $L^{A[S^{-1}]}$. Since $A[S^{-1}]$ is hereditary, every ideal in $A[S^{-1}]$ is projective as an $A[S^{-1}]$-module. Hence, the hypotheses of Corollary \ref{gen thm cor} are satisfied, yielding the isomorphism $L^{A[S^{-1}]}\cong \bigotimes_{n>1} \Rees_{A[S^{-1}]}^{2n-2}\left( I^{A[S^{-1}]}_{n}\right)$. Finally, the Rees algebra functor commutes with localization; this is straightforward to prove from the definition of the Rees algebra, but readers who would rather just see this claim appear in a trustworthy source can find it in the documentation \cite[page 4]{eisenbudrees} for David Eisenbud's Rees package for Macaulay2, for example. This yields the isomorphism
\begin{align*}
 \bigotimes_{n>1} \Rees_{A[S^{-1}]}^{2n-2}\left( I^{A[S^{-1}]}_{n}\right)
 &\cong \left( \bigotimes_{n>1} \Rees_{A}^{2n-2}\left( I^{A}_{n}\right)\right)[S^{-1}],
\end{align*}
where the tensor product on the left is taken over $A[S^{-1}]$, and the tensor product on the right is taken 
over $A$.

An entirely analogous argument works with $\Sym$ in place of $\Rees$. The required fact that $\Sym$ commutes with localization can be proven directly, or it can be taken as a consequence of \cite[Proposition 5.2(iii)]{MR2418168}, for instance.
\end{proof}

\begin{corollary}\label{structure cor in hereditary case}
Let $A$ be a commutative ring and $S$ a set of prime numbers such that
$A$ and $S$ satisfy the assumptions 
of Theorem \ref{main thm in hereditary case}. Then 
the following statements are all true:
\begin{itemize}
\item $L^A[S^{-1}]$ is a commutative graded sub-$A$-algebra of the polynomial algebra $A[S^{-1}][x_1, x_2, \dots ]$, with $x_i$ in degree $2i$.
\item $L^A[S^{-1}]$ is not Noetherian, but for every integer $n$, the sub-$A$-algebra of $L^A[S^{-1}]$ generated by all elements of 
degree $\leq n$ is Noetherian.
\item If $A[S^{-1}]$ is an integral domain, then the underlying $A[S^{-1}]$-module of $L^A[S^{-1}]$ is torsion-free.
\item 
If $U_0^A(n)[S^{-1}]$ is trivial for all $n>1$, then 
$L^A[S^{-1}]$ is polynomial by the fundamental comparison condition, so $L^A[S^{-1}] \cong A[S^{-1}][x_1, x_2, \dots ]$.
\item {\bf All formal module buds extend:} Every formal $A$-module $n$-bud over a commutative $A[S^{-1}]$-algebra extends to a formal $A$-module.
\item {\bf All formal modules lift:}
If $R$ is a commutative $A[S^{-1}]$-algebra and $I$ is an ideal of $R$, then every formal $A$-module over $R/I$ is the modulo-$I$ reduction
of a formal $A$-module over $R$. 
\end{itemize}
\end{corollary}
\begin{proof}
Theorem \ref{main thm in hereditary case} together with Corollary \ref{lifting and extensions}. To give a bit more detail about the fourth bullet point: if $U^A_0(n)[S^{-1}]$ vanishes for all $n>1$, then by Corollary \ref{u-homology and pirashvili cor}, the $U$-homology groups $U^A_i(n)[S^{-1}]$ vanish for all $i\geq 0$. Proposition \ref{fundamental functional main properties ii} then ensures that $L^{A[S^{-1}]}$ is polynomial by the fundamental comparison. 
\end{proof}

\begin{corollary}
Let $K/\mathbb{Q}$ be a finite field extension with ring of integers $A$. Then every formal $A$-module $n$-bud extends to a formal $A$-module.
Furthermore, if $R$ is a commutative $A$-algebra and $I$ is an ideal in $R$, then every formal $A$-module over $R/I$ is the reduction modulo $I$ of a formal $A$-module over $R$.
\end{corollary}

\section{Computations of $L^A$ for certain classes of ring $A$}\label{computations section}

\subsection{Number rings}\label{number rings subsection}

\begin{theorem}\label{number ring iso 0}
Let $A$ be the ring of integers in a finite field extension $K/\mathbb{Q}$,
let $1, \alpha_1, \dots , \alpha_j$ be a $\mathbb{Z}$-linear basis for $A$,
and for each integer $n>1$, let $J^A_{n}$ be the ideal $(\nu(n), \alpha_1^n-\alpha_1, \alpha_2^n - \alpha_2, \dots ,\alpha_j^n - \alpha_j)$ of $A$.
Let $P$ denote the set of integers $>1$ which are prime powers, and let $R$ denote the set of integers $>1$ which are not prime powers.
Then we have an isomorphism of commutative graded $A$-algebras:
\begin{align*}
 L^A
  &\cong \left(\bigotimes^{n\in P} \Rees_A^{2n-2}(J^A_{n})\right)
 \otimes_A A[x_{n-1}: n\in R] ,
\end{align*}
with $x_{n-1}$ in degree $2(n-1)$, and with the tensor products all taken over $A$.
\end{theorem}
\begin{proof}
We use Theorem \ref{main thm in hereditary case}.
If $n$ is not a prime power, then the ideal $I^A_n$ is principal, and hence $\Rees_A^{2n-2}(I^A_{n}) \cong A[x_{n-1}]$.
If $n = p^m$ for some prime number $p$, then
recall that $I^A_{n}$ is the ideal generated by $p$ and by all elements of the form $a^{p^m}-a$ with $a\in A$. One checks easily that,
if an ideal contains $p$ as well as $\alpha^{p^m}-\alpha$ for every element $\alpha$ in some $\mathbb{Z}$-linear basis for $A$,
then that ideal contains $a^{p^m}-a$ for all $a\in A$.
Hence, $J^A_{n} = I^A_{n}$.
\end{proof}

Some (but not all) number fields have the property that their ring of integers can be written in the form $A = \mathbb{Z}[\alpha]$ for some element $\alpha$. Such number fields are said to be {\em monogenic}. Remark \ref{description of ideals I_n} gives us an even more compact description of $L^A$ in that case:
\begin{corollary}\label{number ring iso}
Let $A = \mathbb{Z}[\alpha]$ be the ring of integers in a monogenic finite extension $K/\mathbb{Q}$,
and for each integer $n>1$, let $J^A_n$ be the ideal $(\nu(n), \alpha^n-\alpha)$ of $A$.
Then we have an isomorphism of commutative graded $A$-algebras:
\begin{align*}
 L^A
  &\cong \Rees_A^2(J^A_2)\otimes_A \Rees_A^4(J^A_3)\otimes_A 
 \Rees_A^6(J^A_4)\otimes_A \Rees_A^8(J^A_5)\otimes_A \dots .
\end{align*}
\end{corollary}
\begin{proof}
This is just Theorem \ref{number ring iso 0}
together with Remark \ref{description of ideals I_n}
to get a small set of generators for the ideals $I_n^A$.
\end{proof}

Here is another corollary of Theorem \ref{number ring iso 0}:
\begin{corollary}\label{number ring iso 2}
Let $A$ be the ring of integers in a finite extension $K/\mathbb{Q}$. 
Let $P$ denote the set of integers $>1$ which are prime powers, and let $R$ denote the set of integers $>1$ which are not prime powers.
Then we have an isomorphism of commutative graded $A$-algebras:
\begin{align*}
 L^A
  &\cong \left( \bigotimes^{n\in P} A[x_{n-1},y_{n-1}]/\left(f_{n-1}(x_{n-1},y_{n-1})\right)\right)
 \otimes_A A[x_{n-1}: n\in R] ,
\end{align*}
for some set of polynomials $\{ f_{n-1}\}_{n\in P}$, with each $f_{n-1}\in A[x,y]$,
with $x_{n-1}$ and $y_{n-1}$ in degree $2(n-1)$, and with all tensor products taken over $A$.
\end{corollary}
\begin{proof}
Every ideal in $A$ can be generated by two elements,
so the ideals $J^A_n$ appearing in Theorem \ref{number ring iso 0}
can each be generated by two elements, and with a single relation
between them. Hence, $\Rees_A(J^A_n) \cong A[x,y]/f(x,y)$ with
$f(x,y)$ the relation between the two generators of $J^A_n$.
\end{proof}

\begin{remark}
It seems to have already been known
to Hazewinkel in 1978 that,
when $A$ is the ring of integers in a finite extension of $\mathbb{Q}$,
the $A$-module $\overline{L}^A_{2n-2}$ is isomorphic to the ideal of
$A$ generated by $\nu(n)$ and by all elements of the form
$a^n-a$. See Example 21.3.3A of \cite{MR506881}, where this is almost
(but not quite) stated in these terms.
The full description of $L^A$ given in Theorem \ref{number ring iso 0} is, on the other hand, new.
\end{remark} 

The rest of this subsection consists of special cases, examples, and observations about them. All of this material is, in a way, routine: it follows from using elementary techniques in algebraic number theory to identify the structures of the ideals $I^A_n$ arising in Theorem \ref{number ring iso 0}, in order to calculate $L^A$ for various number rings $A$. The author finds these examples illuminating, as illustrations of how Theorem \ref{number ring iso 0} is fruitfully applied to make calculations, what kinds of observations fall out easily from the calculations, and what kinds of phenomena one encounters. Since part of the audience for this paper may consist of topologists (like the author) who may not be accustomed to arguments involving class groups or using reciprocity laws to deduce the pattern of splitting of primes in a number field, we have elected to leave in some of the details of how these arguments go, although they are very standard arguments in number theory. The reader who does not care for examples will want to skip the rest of this subsection.

\begin{theorem}\label{quadratic case}
Let $K$ be a quadratic extension of the rational numbers,
and let $A = \mathbb{Z}[\alpha]$ be the ring of integers of $K$.
Let $\Delta$ denote the discriminant of $K/\mathbb{Q}$.
Let $R$ be the set of prime numbers $p$ which divide $\Delta$
and which have the property that $I^A_{p^m} = (p, \alpha^{p^m}-\alpha)$ is nonprincipal for some positive integer $m$,
and let $S$ be the set of integers $>1$ which are
not powers of primes contained in $R$.
Then we have an isomorphism of commutative graded $A$-algebras:
\begin{align*} L^A \cong & A[ x_{n-1}: n\in S]  \\
    &\otimes_A\bigotimes^{p\in R} \left( \Rees_A^{2p-2}(I^A_{p}) \otimes_A \Rees_A^{2p^2-2}(I^A_{p^2}) \otimes_A \Rees_A^{2p^3-2}(I^A_{p^3}) \otimes_A \dots \right), \end{align*}
with each polynomial generator $x_{n-1}$ in degree $2(n-1)$, and with all tensor products taken over $A$.

Consequently, 
we have an isomorphism of commutative graded $A[R^{-1}]$-algebras:
\begin{align*}
 L^A[R^{-1}]
  &\cong A[R^{-1}][x_1, x_2, \dots ],
\end{align*}
with each $x_n$ in degree $2n$.
\end{theorem}
\begin{proof}
If we can show that the ideal $I_{n}^A$ is principal for all integers $n\in S$, then $I_{n}^A$ is projective as an $A$-module, so the claims all follow from Theorem \ref{number ring iso 0}. If $n$ is not a prime power, then $1 = \nu(n) \in I_{n}^A$, so $I_{n}^A$ is certainly principal in that case. Consequently,  assume that $n = p^m$ for some prime $p\notin R$. Our goal is to show that the ideal $I_{p^m}^A$ is principal.

Every quadratic extension $K$ of $\mathbb{Q}$ can be written as $K = \mathbb{Q}(\sqrt{d})$ for some square-free integer $d$. We now break into three cases.
\begin{description}[leftmargin=*]
\item[If $d$ is congruent to $2$ or $3$ modulo $4$]
Then $A = \mathbb{Z}[ \sqrt{d}]$, and by Remark \ref{description of ideals I_n},
$I^A_{p^m} = (p, \sqrt{d} - \sqrt{d}^{\, p^m})$.
The assumption about the congruence class of $d$ entails that the primes dividing the discriminant $\Delta$ are $2$ and the primes dividing $d$. Since we only need check that $I^A_{p^m}$ is principal for primes $p$ which do not divide $\Delta$, we may assume that $p\neq 2$ and that $p$ does not divide $d$. 
Consequently, the ideal $I^A_{p^m}$ contains 
$(\sqrt{d} - \sqrt{d}^{\, p^m})^2 
  = d (1 - d^{\frac{p^m-1}{2}})^2.$

Now we consider two possibilities: either $p$ divides  $1 - d^{\frac{p^m-1}{2}}$, or it doesn't.
If $p$ {\em doesn't} divide $1 - d^{\frac{p^m-1}{2}}$,
then $p$ is coprime to $d (1 - d^{\frac{p^m-1}{2}})^2$
and hence $I^A_{p^m} = (1)$, which is certainly a principal ideal.
On the other hand, if $p$ {\em does} divide $1 - d^{\frac{p^m-1}{2}}$,
then \[ I^A_{p^m} = (p, \sqrt{d} - \sqrt{d}^{\, p^m}) = (p, \sqrt{d}(1 - d^{\frac{p^m-1}{2}})) = (p),\] which is again principal.

We conclude that, if $d\equiv 2$ or $3$ modulo $4$, then $I_{p^m}^A$ is principal.
\item[If $p$ is odd and $d$ is congruent to $1$ modulo $4$]
If we write $\alpha = \frac{1}{2} + \frac{\sqrt{d}}{2}$, then $A = \mathbb{Z}[\alpha]$,
and the primes dividing the discriminant $\Delta$ are exactly the primes dividing $d$.
It is still the case that $\sqrt{d}\in A$, even though $A$ is not equal
to $\mathbb{Z}[\sqrt{d}]$. Consequently, the ideal $I^A_{p^m}$ still contains 
$(\sqrt{d} - \sqrt{d}^{p^m})^2$. From here the argument for the principalness of $I^A_{p^m}$ is exactly the same as in the case $d\equiv 2$ or $3$ modulo $4$.
\item[If $p=2$ and $d$ is congruent to $1$ modulo $4$]
The minimal polynomial of $\alpha\in A$ is
$\alpha^2 - \alpha + \frac{1-d}{4} = 0$,
so modulo $2$,
we have
\begin{align*}
 \alpha-\alpha^{2^m} 
  &\equiv \alpha - \left(\alpha + \frac{d-1}{4}\right)^{2^{m-1}} \\
  &\equiv -\left(\frac{d-1}{4}\right)^{2^{m-1}} + \alpha - \alpha^{2^{m-1}} \\
  &\equiv -\left(\frac{d-1}{4}\right)^{2^{m-1}} + \alpha - \left( \alpha + \frac{d-1}{4}\right)^{2^{m-2}} \\
  &\equiv \dots \\
  &\equiv - \sum_{i=0}^{m-1} \left(\frac{d-1}{4}\right)^{2^{i}},
\end{align*}
which is an integer, consequently is either $0$ or $1$
modulo $2$. If it is congruent to $0$ modulo $2$,
then $A/I^A_{2^m} \cong A/(2)$ and hence 
$I^A_{2^m} = (2)$,
which is principal. On the other hand, if it is congruent to
$1$ modulo $2$,
then $A/I^A_{2^m} \cong 0$ and hence 
$I^A_{2^m} = (1)$, which is still principal.
\end{description}
We conclude that, when $p\notin R$, the ideal $I^A_{p^m}$ is principal, as desired.
\end{proof}

\begin{corollary}\label{quadratic inverted discriminant corollary 1} Let $K$ be a quadratic extension of the rational numbers, let $A$ be the ring of integers of $K$, and let $\Delta$ denote the discriminant of $K/\mathbb{Q}$. Then we have an isomorphism of commutative graded $A[R^{-1}]$-algebras: \begin{align*}
 L^A\left[\frac{1}{\Delta}\right]
  &\cong A\left[\frac{1}{\Delta}\right]\left[x_1, x_2, \dots \right],
\end{align*}
with each $x_i$ in degree $2i$.
\end{corollary}


\begin{example}\label{square root of -5}
Let $A$ be the ring of integers in $\mathbb{Q}(\sqrt{-5})$.
By Theorem \ref{quadratic case}, the only primes $p$
such that $I^A_{p^m}$ is possibly nonprincipal are those primes $p$
that ramify in $A$, i.e., $2$ and $5$. Let $\alpha$ denote
a square root of $-5$ in $A$. By direct computation one
finds that $I^A_{2^m} = (2,\alpha - 1)$ for all $m$, which is nonprincipal,
and that $I^A_{5^m} = (\alpha)$ for all $m$, which is of course principal.
Consequently:
\begin{align*} L^A &\cong \frac{A[x_1,y_1, x_3,y_3,x_{7},y_{7},x_{15},y_{15}, \dots ]}{\left( (\alpha - 1)x_{2^i-1} - 2y_{2^i-1} \ \forall i\geq 1 \right)} \otimes_A A[ x_j : j\neq
2^i-1]\end{align*}
where each $x_j$ and each $y_j$ is in degree $2j$. 
\end{example}

\begin{remark}
Corollary \ref{quadratic inverted discriminant corollary 1} 
does {\em not} remain true if we simply remove the word ``quadratic'' from its statement; it is not the case that, for the ring of integers $A$ in an arbitrary finite extension $K/\mathbb{Q}$, the ring $L^A$ becomes polynomial
after inverting the discriminant of $K/\mathbb{Q}$. For example, the cubic field $\mathbb{Q}(\sqrt[3]{7})$ has the property that, in its ring of integers $\mathbb{Z}[\sqrt[3]{7}]$, the ideals \[ I^{\mathbb{Z}[\sqrt[3]{7}]}_2,I^{\mathbb{Z}[\sqrt[3]{7}]}_3,I^{\mathbb{Z}[\sqrt[3]{7}]}_5,I^{\mathbb{Z}[\sqrt[3]{7}]}_{11},I^{\mathbb{Z}[\sqrt[3]{7}]}_{17}, I^{\mathbb{Z}[\sqrt[3]{7}]}_{23}, I^{\mathbb{Z}[\sqrt[3]{7}]}_{47},I^{\mathbb{Z}[\sqrt[3]{7}]}_{53},I^{\mathbb{Z}[\sqrt[3]{7}]}_{59},\] and probably $I^{\mathbb{Z}[\sqrt[3]{7}]}_p$ for many other $p$, are nonprincipal, as one can verify with a few lines of Sage \cite{sagemath66} or Magma \cite{MR1484478}. However, $3$ and $7$ are the only primes dividing the discriminant of $\mathbb{Q}(\sqrt[3]{7})/\mathbb{Q}$, so inverting the discriminant does not make $L^{\mathbb{Z}[\sqrt[3]{7}]}$ isomorphic to a polynomial algebra. There is no finite collection of primes one can invert to make $L^{\mathbb{Z}[\sqrt[3]{7}]}$ isomorphic to a polynomial algebra. This is a consequence of the Galois closure of $\mathbb{Q}(\sqrt[3]{7})/\mathbb{Q}$ being a {\em nonabelian} extension of $\mathbb{Q}$, so there is no straightforward reciprocity law which establishes that how primes split, and whether they are principal, is governed simply by the congruence class of the prime modulo some conductor. A fuller discussion of these ideas requires a much longer excursion into number theory than fits within the scope of this paper.
\end{remark}

\begin{remark}
In the proof of Theorem \ref{quadratic case}, we show that the ideal $I^A_{p^m}$ in $A$ is principal for all $p$ not ramifying in $A$.
It is worth mentioning that this means that $I^A_{p^m}$ is often principal even when the factors of $(p)$ in $A$ are not principal.
For example, in the case of Example \ref{square root of -5} (i.e., $A = \mathbb{Z}[\sqrt{-5}]$), the class number is $2$, and the prime numbers $3,7,23,43,47,67,83,103$, and many others, all split as products of distinct nonprincipal primes. But $I^A_{p^m} \subseteq \mathbb{Z}[\sqrt{-5}]$ is still principal for
those primes $p$ and for all positive integers $m$.

Something similar happens in the case of Theorem \ref{fourth root of -18}, i.e., $A$ the ring of integers in $\mathbb{Q}(\sqrt[4]{-18})$:
the prime numbers $17, 41, 59, 107, 137, 179, 227$, and many others split in $A$ and have nonprincipal prime factors, but $I^A_{p^m} \subseteq A$ is still principal for
those primes $p$ and for all positive integers $m$.
\end{remark}

In 21.3.3A of \cite{MR506881}, Hazewinkel explains that the extension 
$K = \mathbb{Q}(\sqrt[4]{-18})$ of $\mathbb{Q}$ has the property that
the ideal $I^{A}_2$ of $A$ 
is nonprincipal,
and consequently $L^{A}$ could not be a polynomial $A$-algebra.
Hazewinkel does not attempt a computation of $L^A$, however.
We now compute $L^A$ explicitly:
\begin{theorem} \label{fourth root of -18}
Let $K = \mathbb{Q}(\sqrt[4]{-18})$,
and let $A$ be the ring of integers of $K$.
Let $S$ denote the set of all integers $>1$
which are not powers of $2$ or of $3$.
Then we have an isomorphism of commutative graded $A$-algebras
\begin{align*} 
 L^A 
  & \cong A[x_{n-1}: n\in S] \otimes_A 
    A[x_1,y_1]/(2x_1-(\alpha^2-\alpha)y_1) \\ 
  & \ \ \ \ \ \otimes_A 
    \bigotimes^{m\geq 2} \left( A[x_{2^m-1},y_{2^m-1}]/(2x_{2^m-1}-\alpha y_{2^m-1})\right) \\
  & \ \ \ \ \ \otimes_A
    \bigotimes^{m\geq 1} \left( A[x_{3^m-1},y_{3^m-1}]/(3x_{3^m-1}-\alpha y_{3^m-1})\right),
\end{align*}
where $\alpha = \sqrt[4]{-18}\in A$,
where the polynomial generators $x_{i}$ and $y_i$ are 
in degree $2i$, and where all tensor products are taken over $A$.

Consequently, 
we have an isomorphism of commutative graded $A[\frac{1}{6}]$-algebras:
\begin{align*}
 L^A\left[\frac{1}{6}\right]
  &\cong A\left[\frac{1}{6}\right]\left[x_1, x_2, \dots \right],
\end{align*}
with each $x_i$ in degree $2i$.
\end{theorem}
\begin{proof}
Write $\alpha$ for $\sqrt[4]{-18}$.
It is routine to calculate that the ring of integers of $A$ is the ring of $\mathbb{Z}$-linear combinations
of $1,\alpha,\frac{1}{3} \alpha^2,$ and $\frac{1}{3} \alpha^3$.
Consequently, 
\[ I^A_{p^m} = \left(p,\ \alpha - \alpha^{p^m},\ \frac{1}{3} \alpha^2 - \left(\frac{1}{3} \alpha^2\right)^{p^m},\ \frac{1}{3} \alpha^3 - \left(\frac{1}{3} \alpha^3\right)^{p^m}\right) .\]
If $p\neq 3$, then $\frac{1}{3} \alpha^2 - \left(\frac{1}{3} \alpha^2\right)^{p^m}$ and $\frac{1}{3} \alpha^3 - \left(\frac{1}{3} \alpha^3\right)^{p^m}$ are already in the ideal generated by $p$ and $\alpha - \alpha^{p^m}$. Hence, if $p\neq 3$, then $I^A_{p^m} = (p, \alpha - \alpha^{p^m})$.

We will handle the primes $p=2$ and $3$ shortly, but for now, let $p>3$ be a prime number and let $m$ be a positive integer. We claim that the ideal $I^A_{p^m}$ is principal.
To prove this claim, we break into cases:
\begin{description}
\item[If $p^m-1$ is congruent to $0$ modulo $4$]
Then:
\begin{align}
 \alpha - \alpha^{p^m} 
\label{equality 1190982}  &= \alpha\left(1 - (-18)^{\frac{p^m-1}{4}}\right) ,
  \end{align}
and consequently
$\left(\alpha - \alpha^{p^m}\right)^4
= -18\left(1 - (-18)^{\frac{p^m-1}{4}}\right)^4.$

If $p$ {\em does not} divide $1 - (-18)^{\frac{p^m-1}{4}}$ then $p$ is coprime to 
the integer $\left(\alpha - \alpha^{p^m}\right)^4\in I^A_{p^m}$.
Hence, $I^A_{p^m}$ contains two coprime integers, hence
$I^A_{p^m} = (1)$, which is principal. If $p$ instead {\em does} divide $1 - (-18)^{\frac{p^m-1}{4}}$,
then $\alpha-\alpha^{p^m} \in (p)$, by equality \eqref{equality 1190982},
so $I^A_{p^m} = (p)$, which is again principal.
\item[If $p^m-1$ is congruent to $2$ modulo $4$]
We have 
\begin{align}
 \alpha - \alpha^{p^m} 
\label{equality 1190983}  &= \alpha\left(1 - 3\sqrt{-2} (-18)^{\frac{p^m-3}{4}}\right) .
  \end{align}
Since $p>3$, the ideal $(\alpha)$ is coprime to $(p)$. Consequently, $(p,\alpha - \alpha^{p^m}) = (p,1 - 3\sqrt{-2} (-18)^{\frac{p^m-3}{4}})$, which is extended up from the subring $\mathbb{Z}[\sqrt{-2}]$ of $A$. The class number of $\mathbb{Q}(\sqrt{-2})$ is $1$, so $I^A_{p^m} = (p,1 - 3\sqrt{-2} (-18)^{\frac{p^m-3}{4}})$ is principal. 
\end{description}
Since $p^m$ is odd, the cases $p^m-1 \equiv 0$ and $p^m - 1\equiv 2$ modulo $4$ are all the possible cases.
Hence, if $p>3$, then $I^A_{p^m}$ is principal.

Now for the primes $p=2$ and $p=3$.
For $p=2$, we have $I_2^A = (2,\alpha^2-\alpha)$, which is nonprincipal by elementary calculation. If $m>1$, then
\[ I_{2^m}^A = (2,\alpha-\alpha^{2^m}) = (2,\alpha - (-18)^{2^{m-2}}) = (2,\alpha),\] 
so $I_{2^m}^A$ is nonprincipal as well.

For $p=3$, we claim that $I^A_{3^m} = (3,\alpha)$ for all $m\geq 1$, 
which is again nonprincipal. The proof is as follows: first, since $-2 \equiv 1$ modulo $3$, clearly $1 - (-2)^{\frac{3^m-1}{2}}$ is divisible by $3$. Consequently, we have equalities in $A$
\begin{align*}
 \frac{\alpha^2}{3} - \left(\frac{\alpha^2}{3}\right)^{3^m}
  &= \frac{\alpha^2}{3}\left( 1 - \left( \frac{\alpha^2}{3}\right)^{3^m-1}\right) \\
  &= \sqrt{-2}\left( 1 - (-2)^{\frac{3^m-1}{2}}\right) ,
\end{align*}
hence equalities of ideals
\begin{align*}
 \left(3,\alpha-\alpha^{3^m},\frac{\alpha^2}{3} - \left(\frac{\alpha^2}{3}\right)^{3^m}\right)
  &= \left(3,\alpha \left(1 - (3\sqrt{-2})^{\frac{3^m-1}{2}}\right),\sqrt{-2}\left(1 - (-2)^{\frac{3^m-1}{2}}\right)\right) \\
  &= (3,\alpha).
\end{align*}
It is elementary to verify that $\frac{\alpha^3}{3} - \left(\frac{\alpha^3}{3}\right)^{3^m}\in (3,\alpha)$ as well. This yields the desired equality of ideals $I^A_{3^m} = (3,\alpha)$.

The theorem as stated now follows from Theorem \ref{number ring iso 0}.
\end{proof}

\subsection{Group rings}

Let $C_m$ be the cyclic group with $m$ elements. In Theorem \ref{group ring computation} we compute the classifying ring $L^{\mathbb{Z}[C_m]}$ of formal $\mathbb{Z}[C_m]$-modules, after inverting $m$, so that $\mathbb{Z}[C_m]\left[\frac{1}{m}\right]$ is hereditary and Theorem \ref{main thm in hereditary case} applies. The resulting ring $L^{\mathbb{Z}[C_m]}[\frac{1}{m}]$ is not a polynomial algebra but nevertheless has a tractable presentation.

\begin{theorem}\label{group ring computation}
Let $C_m$ be the cyclic group of order $m$.
Let $P$ be the set of integers~$>1$ which are
prime powers relatively prime to $m$.
Let $S$ be the set of integers $>1$ not contained in $P$.
Write $R$ for the group $\mathbb{Z}\left[\frac{1}{m}\right]$-algebra $\mathbb{Z}\left[\frac{1}{m}\right][C_m]$ of $C_m$.
Then we have an isomorphism of graded $\mathbb{Z}\left[\frac{1}{m}\right]$-algebras
\begin{align*} L^{\mathbb{Z}[C_m]}\left[\frac{1}{m}\right] &\cong 
 \bigotimes^{n\in P} \left( R[x_{n-1},y_{n-1}]/(\nu(n)x_{n-1} - (1 - \sigma)y_{n-1})\right) \\ &\ \ \ \otimes_{R}
 R[x_{n-1}: n\in S],\end{align*}
where $\sigma$ denotes a generator of $C_m$, where  
the polynomial generators $x_{n-1}$ and $y_{n-1}$ are each 
in degree $2(n-1)$, and where all tensor products are taken over $R$.
\end{theorem}
\begin{proof}
This is another special case of Theorem \ref{main thm in hereditary case},
since $\mathbb{Z}[C_m]$ is torsion-free and since
the ring $R$ is hereditary.
If $p$ divides $m$, then 
clearly the ideal $I^{\mathbb{Z}[C_m]}_{p^i}$ becomes principal after inverting $m$,
hence $\Rees_{\mathbb{Z}[C_m]}(I^{\mathbb{Z}[C_m]}_{p^i})\left[\frac{1}{m}\right] \cong R[x]$ if $p$ divides $m$.

Now suppose that $p$ does not divide $m$. By Remark \ref{description of ideals I_n}, the two elements $p$ and $\sigma^{p^i}-\sigma$ suffice to generate the ideal $I_{p^i}^{\mathbb{Z}[C_m]}$.
The projection $\mathbb{Z}[C_m] \rightarrow \mathbb{Z}[C_m]/(p, \sigma^{p^i} - \sigma)$ sends $p$ to zero and $\sigma$ to $1$, i.e., the kernel of the projection is
the ideal $(p, \sigma-1)$. Hence, $I^{\mathbb{Z}[C_m]}_{p^i} = (p, \sigma-1)$.
Now the claim follows from Theorem \ref{main thm in hereditary case} and the standard presentation for the Rees algebra of a two-generator ideal, as in the proof of Corollary \ref{number ring iso 2}.
\end{proof}

\begin{remark}
Suppose $G$ is a finite abelian group. 
In \cite{MR1770613} (see also \cite{MR1856028} for a nice survey) a theory of ``$G$-equivariant formal group''
is developed, which is designed to admit a classifying ring
$L^G$ for $G$-equivariant formal groups with a canonical comparison map
with the $G$-equivariant
complex bordism ring $MU^G_*$. It is a highly nontrivial fact that the comparison map is an isomorphism: this was {\em Greenlees' conjecture}, proven in the case $G = \mathbb{Z}/2\mathbb{Z}$ in \cite{MR3886163} and for an arbitrary compact abelian Lie group $G$ in \cite{MR4413745}.

A $G$-equivariant formal group is a more complicated gadget than just 
a formal group $F$ equipped with a choice of group homomorphism 
$G \rightarrow \Aut(F)$. Even the definition of a $G$-equivariant
formal group is rather involved, and since it is only tangentially related to the present paper, we refer the reader to \cite{MR1770613} for the definition. One expects that
$G$-equivariant formal groups ought to have some relationship with
the simple notion of a formal group equipped with an action by the group $G$, which is nearly the same thing as a formal $\mathbb{Z}[G]$-module. Clearly, to specify the structure map $\rho: \mathbb{Z}[G] \rightarrow \End(F)$ of a formal $\mathbb{Z}[G]$-module, we could just as well have specified a group homomorphism $G \rightarrow \Aut(F)$; the only point to mention here is
the tangency axiom in the definition of a formal module, i.e., that
$\rho(g)(X) \equiv gX \mod X^2$, meaning that we need to have an action of
$G$ on the coefficient ring of $F$. But if we begin with a group homomorphism $G \rightarrow \Aut(F)$,
we can typically choose an action of $G$ on the coefficient ring of $F$ so that the tangency axiom is satisfied. Hence, the distinction between a $\mathbb{Z}[G]$-module and a formal group with an action by $G$ is slight. 

The author hopes that Theorem \ref{group ring computation} will turn out to be useful in a comparison between the moduli of formal $\mathbb{Z}[G]$-modules and the moduli of $G$-equivariant formal group laws, after inverting the order of $G$. Such a comparison is not within the scope of this paper, though.
\end{remark}

\def\cprime{$'$} \def\cprime{$'$} \def\cprime{$'$} \def\cprime{$'$}

\end{document}